\documentclass{article}
\usepackage{amssymb}
\usepackage{amsmath}
\usepackage{amsthm}
\usepackage{graphicx}
\usepackage{subfig}
\usepackage{float}
\usepackage{cite}
\usepackage{caption}
\usepackage{float} 
\usepackage{color}
\usepackage{mathrsfs}

\usepackage[left=1 in,right=1 in,top=1 in,bottom=1 in]{geometry}

\DeclareMathAlphabet{\mathscr}{OT1}{pzc}{m}{it}

\newtheorem{remark}{Remark}
\newtheorem{corollary}{Corollary}
\newtheorem{lemma}{Lemma}
\newtheorem{theorem}{Theorem}
\newtheorem{example}{Example}

\title{
	\bf{The existence of quasi-periodic invariant tori and double Hopf bifurcation of van der Pol's oscillator with delayed feedback}
	\author{
		Xuemei Li, Bochao Yu\thanks{Corresponding author.}
		\setcounter{footnote}{-1}
		\footnote{E-mail addresses: lixuemei\_1@sina.com (X. Li), yubochao\_1@hunnu.edu.cn (B. Yu).}\\\\
		\small Key Laboratory of Computing and Stochastic Mathematics (Ministry of Education),\\
		\small School of Mathematics and Statistics, Hunan Normal University, Changsha, Hunan 410081, P. R. China.
	}
}

\date{}

\begin{document}

\maketitle

\textbf{Abstract}
The double Hopf bifurcation and the existence of quasi-periodic invariant tori in a delayed van der Pol's oscillator are considered by regarding the damped coefficient and the delay as bifurcation parameters.
Applying the center manifold reduction and the normal form method, we derive the normal form near the critical point and analyse the existence of invariant 2-tori and 3-tori for the truncated normal form.
Furthermore, the effect of higher-order terms on these invariant tori are investigated by a KAM theorem, and it is proved that in a sufficiently small neighborhood of the bifurcation point, the delayed van der Pol's oscillator has quasi-periodic invariant 2-tori and 3-tori for most of the parameter set where its truncated normal form possesses quasi-periodic invariant 2-tori and 3-tori, respectively.

\textbf{Keywords}
van der Pol's oscillator,
Double Hopf bifurcation,
Normal form,
Center manifold reduction,
Quasi-periodic invariant torus.


\section{Introduction}
Balthasar van der Pol \cite{L1} presented the oscillator model
\begin{align*}
	\ddot{x}(t)+\varepsilon(x^{2}(t)-1)\dot{x}(t)+x(t)=f(t),
\end{align*}
where $\varepsilon>0$ is the damped coefficient, and $f(t)$ is the forcing.
Originally it is a model for the electrical circuit with a triode valve, and has evolved into one of the most celebrated equations in the study of nonlinear dynamics \cite{L17,L18,L19}.
Van der Pol has found that the unforced system has an unstable trivial equilibrium and contains a stable limit cycle for all values of $\varepsilon>0$.
When $f(t)$ is periodic or quasi-periodic, there are much more complicated phenomena in the system, and it can exhibit chaotic behaviour.
\par
In recent years a growing interest has been witnessed in the effect of delayed feedback on dynamics \cite{L20,L21}.
There have been much study of van der Pol's oscillator with delayed feedback, in particular, from bifurcation point of view, see \cite{L3,L4,L5,L6,L7,L8,L9}.
Atay \cite{L2} used a linear delayed feedback as the forcing
\begin{align}\label{1.2}
	\ddot{x}(t)+\varepsilon(x^{2}(t)-1)\dot{x}(t)+x(t)=\varepsilon k x(t-\tau),
\end{align}
and studied the behaviour of the limit cycle.
By the normal form method and the center manifold reduction, Wei and Jiang \cite{L3} showed that the system \eqref{1.2} undergoes a Hopf bifurcation and discussed the direction and stability of Hopf bifurcation.
\par
Introducing a nonlinear delayed feedback to the van der Pol's oscillator
\begin{align}\label{fulleq}
	\ddot{x}(t)+\varepsilon(x^{2}(t)-1)\dot{x}(t)+x(t)=\varepsilon f(x(t-\tau)),
\end{align}
where $f(0)=f^{\prime\prime}(0)=0$, $f^{\prime}(0)=a\neq0$, Jiang and Wei \cite{L4} found that under certain conditions, single, double and triple zero eigenvalues are possible for equilibria as well as a purely imaginary pair with a zero eigenvalue.
Furthermore, some researchers considered corresponding bifurcations for these cases, including the fold and Hopf bifurcation\cite{L4}, Bogdanov-Takens bifurcation\cite{L5}, Zero-Hopf bifurcation \cite{L6,L7} and Triple bifurcation \cite{L8}.
Zhang and Guo \cite{L9} analysed non-semisimple 1:1 resonance Hopf bifurcation by using the normal form approach, and found that there exist zero, one or two small-amplitude periodic solutions.
In the codimension 2 bifurcations, only the case where the characteristic equation has two pairs of simple imaginary roots has not been studied.
\par
In this paper, we consider the double Hopf bifurcation of system \eqref{fulleq} where its characteristic equation has two pairs of simple imaginary roots, and the persistence of invariant tori arising from the double Hopf bifurcation.
First of all, we obtain the sufficient condition that the characteristic equation has two pairs of simple imaginary roots, and analyse the existence of invariant 2-tori and 3-tori for the truncated normal form of system \eqref{fulleq}.
There are some works on double Hopf bifurcation for delayed oscillators \cite{L22,L16,L23,L24}, where the effect of higher-order terms is discussed by numerical simulations.
Since Kuznetsov pointed out in his monograph \cite{L11} that the existence of an invariant torus in the truncated system does not guarantee that the full system will also have one, the effect of higher-order terms on these invariant tori should be taken into account.
Then, we aim to investigate whether 2-tori and 3-tori will remain by the KAM method \cite{L12,L13,L14,L15} if we add higher-order terms to the truncated normal form.
We will prove that in a sufficiently small neighborhood of the bifurcation point, system \eqref{fulleq} has quasi-periodic invariant 2-tori and 3-tori for most of the parameter set where its truncated normal form possesses quasi-periodic invariant 2-tori and 3-tori, respectively.
Finally, some numerical simulations are provided to support our theoretical results.
Throughout this paper, we always assume that $\varepsilon>0$ and the function $f(x)$ is sufficiently smooth in a neighbourhood of $x=0$ and 
\begin{align*}
	f(0)=f^{\prime\prime}(0)=0, \quad f^{\prime}(0)=a \neq 0.
\end{align*}
Denote
\begin{align*}
	f^{\prime\prime\prime}(0)=3!b, \quad f^{(4)}(0)=4!c, \quad f^{(5)}(0)=5!d.
\end{align*}
\par
The rest of this paper is organized as follows.
The distribution of characteristic roots, the derivation of normal form and the dynamic behaviors of the truncated normal form are developed in Section \ref{sec2}.
In Section \ref{sec3}, the persistence of quasi-periodic invariant 2-tori and 3-tori of the truncated system of \eqref{fulleq} is proved by a KAM theorem.
Section \ref{sec4} provides some numerical simulations.
Finally, the Appendix collects the proof of Lemma \ref{lemma2} and some complex coefficients in the normal form.

\section{Double Hopf bifurcation}\label{sec2}
\subsection{The characteristic equation}\label{sec2.1}
In this subsection, the double Hopf bifurcation is investigated by regarding $\varepsilon$ and $\tau$ as bifurcation parameters.
It is clear that the origin is an equilibrium of \eqref{fulleq}, and the associated characteristic equation is
\begin{align}\label{ceq}
	\lambda^{2} - \varepsilon \lambda - \varepsilon a e^{- \lambda \tau} + 1 = 0.
\end{align}
In order to obtain the sufficient condition that the characteristic equation has two pairs of simple imaginary roots, substituting $\lambda=\pm {\rm i}\omega$ into \eqref{ceq} and separating the real and imaginary parts gives
\begin{align}\label{ceqs1}
	\left\{
	\begin{array}{ll}
		\omega=a\sin(\omega\tau),\\
		1-\omega^{2}=\varepsilon a\cos(\omega\tau).
	\end{array}
	\right.
\end{align}
It follows from \eqref{ceqs1} that $\omega$ satisfies
\begin{align}\label{ceqs2}
	\omega^{4}+(\varepsilon^{2}-2)\omega^{2}+1-\varepsilon^{2} a^{2}=0.
\end{align}
Assume
\begin{align}\label{H0}
	0<\varepsilon<\sqrt{2}, \qquad \varepsilon \left| a \right| <1, \qquad \varepsilon^{2}+4a^{2}>4,
\end{align}
then \eqref{ceqs2} has two different positive roots
\begin{align*}
	\omega_{1,2}=\frac{1}{\sqrt{2}} \left( 2-\varepsilon^{2} \pm \varepsilon \sqrt{\varepsilon^{2}-4+4 a^{2}}\right) ^{\frac{1}{2}}, \qquad \omega_{1}>\omega_{2},
\end{align*}
and \eqref{ceqs1} implies
\begin{align*}
	\tau_{1,2}^{j}=\left\{
		\begin{array}{ll}
			\frac{1}{\omega_{1,2}}(\arccos(\frac{1-\omega_{1,2}^{2}}{\varepsilon a})+2j\pi), & a>0,\\
			\frac{1}{\omega_{1,2}}(2\pi-\arccos(\frac{1-\omega_{1,2}^{2}}{\varepsilon a})+2j\pi), & a<0,
		\end{array}
	\right.
	\qquad 
	j=0,1,2,\cdots.
\end{align*}
From the above discussion, we obtain the following lemma.
\begin{lemma}\label{lemma1}
	Assume that \eqref{H0} holds for system \eqref{fulleq}.
	If there is an $\varepsilon_{0}$ such that $\tau_{1}^{m}=\tau_{2}^{n}:=\tau_{0}$ for some $m,n\in\mathbb{N}$, then the characteristic equation \eqref{ceq} has two pairs of simple imaginary roots $\pm{\rm i}\omega_{1}$ and $\pm{\rm i}\omega_{2}$ at $(\varepsilon,\tau)=(\varepsilon_{0},\tau_{0})$.
\end{lemma}
\begin{figure}
	\centering
	\includegraphics[width=100mm]{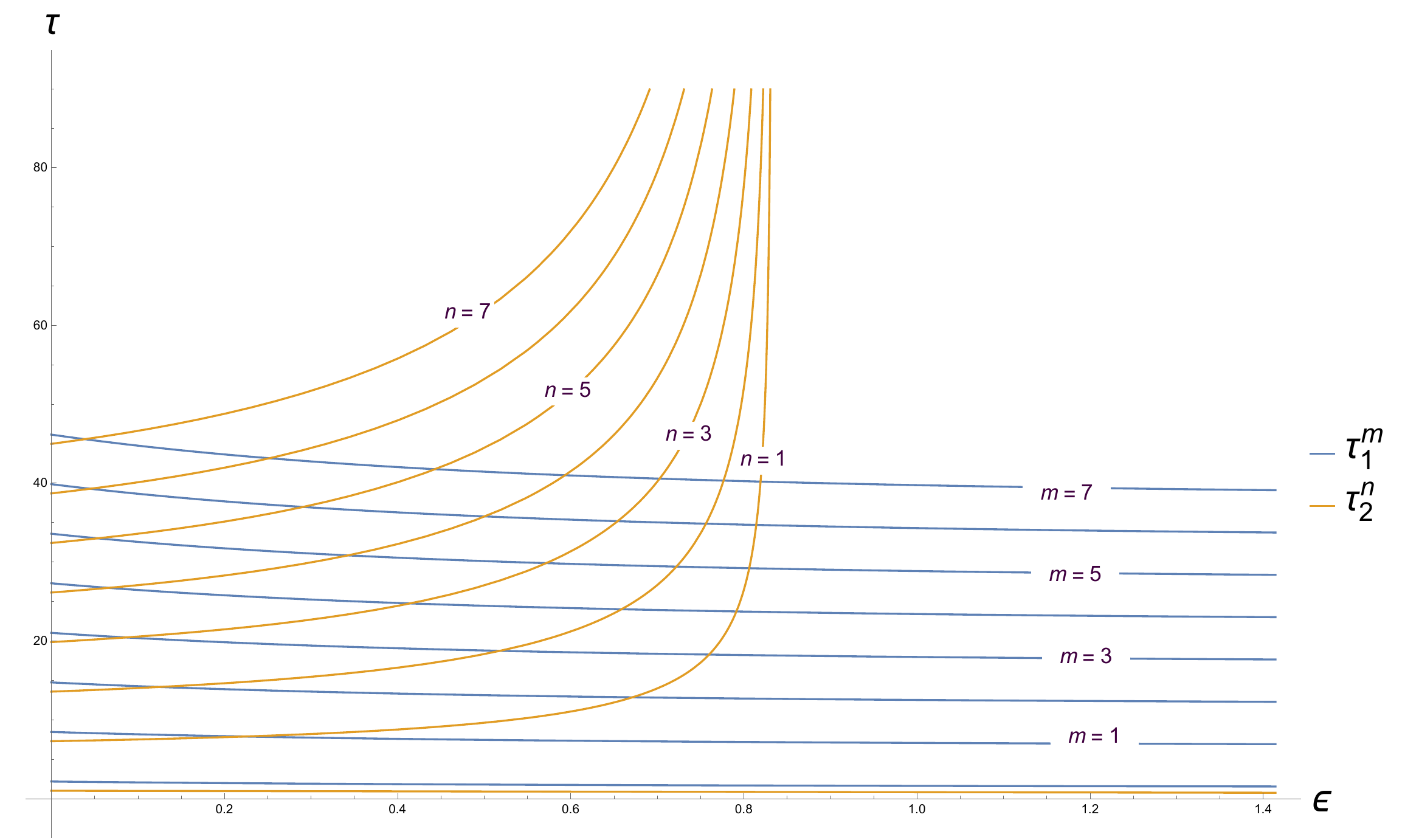}
	\caption{The Hopf bifurcation curves of system \eqref{fulleq} for $a=1.2$ and $m$, $n$ are from $0$ to $7$.}
	\label{figure-existence}
\end{figure}
As shown in Figure \ref{figure-existence}, there are some intersection points of $\tau_{1}^{m}$ and $\tau_{2}^{n}$.
Furthermore, a double Hopf bifurcation occurs at these points.
\par
It is well known that the double Hopf bifurcation may lead to more complex behaviors if the imaginary roots are resonant.
Letting $\pm{\rm i}\omega_{1}$, $\pm{\rm i}\omega_{2}$ be two pairs of imaginary roots of \eqref{ceq}, if $\omega_{1}:\omega_{2}=3:1$, $2:1$ or $1:1$, the double Hopf bifurcation is called strong resonant, which is a codimension-three bifurcation. 
Otherwise, it is called weak resonance or nonresonance, which is a codimension-two bifurcation.
For system \eqref{fulleq}, the 1:1 resonant case only can occur when the third inequality in (6) is replaced by $\varepsilon^{2}+4a^{2}=4$, and this case has been studied in \cite{L9}.
The following result indicates that the other two cases of the strong resonance do not exist.
\begin{lemma}\label{lemma2}
	Under assumption \eqref{H0}, $k\omega_{1} \neq l\omega_{2} $ for all integers $k$ and $l$ that satisfyfy $1 \leq \left| k \right| + \left| l \right| \leq 4$.
\end{lemma}
\begin{remark}
	In fact, under assumption \eqref{H0}, we can obtain that $\omega_{1}$ and $\omega_{2}$ are nonresonant up to order $6$, i.e., $k\omega_{1} \neq l\omega_{2} $ for all integers $k$ and $l$ satisfying $1 \leq \left| k \right| + \left| l \right| \leq 6$.
\end{remark}
The proof of Lemma \ref{lemma2} is seen in Appendix A.1.
\par
Let $\lambda_{1,2}(\varepsilon,\tau)=\mu_{1,2}(\varepsilon,\tau)\pm{\rm i}\omega_{1,2}(\varepsilon,\tau)$ be two roots of \eqref{ceq} near $(\varepsilon_{0},\tau_{0})$, such that $\mu_{1,2}(\varepsilon_{0},\tau_{0})=0$, $\omega_{1,2}(\varepsilon_{0},\tau_{0})=\omega_{1,2}$.
According to the definition of double Hopf bifurcation, the transversality condition
\begin{equation}\label{tc}
	\begin{aligned}
		\Delta=&\Delta(\varepsilon_{0}, \tau_{0})=\det \left. \frac{\partial(\mu_1,\mu_2)}{\partial(\varepsilon, \tau)} \right|_{(\varepsilon,\tau)=(\varepsilon_{0},\tau_{0})}\\
		=&\frac{\varepsilon_{0} (\omega_{2}^{2}-\omega_{1}^{2}) (2 a^{2} \tau_{0}+a^2 \tau _0 \varepsilon _0^2-2 a^2 \varepsilon _0-2 \tau _0+\varepsilon _0)}{\left(a^2 \tau _0^2 \varepsilon _0^2+4 \omega _1^2-2 \tau _0 \varepsilon _0 \omega _1^2 -2 \tau _0 \varepsilon _0+\varepsilon _0^2\right) \left(a^2 \tau _0^2 \varepsilon _0^2+4 \omega _2^2-2 \tau _0 \varepsilon _0 \omega _2^2 -2 \tau _0 \varepsilon _0+\varepsilon _0^2\right)} \neq 0
	\end{aligned}
\end{equation}
should be taken into account.
In summary, we have the following result.
\begin{theorem}\label{theorem1}
	Suppose that system \eqref{fulleq} satisfies \eqref{H0}. 
	If there exists an $\varepsilon_{0}$ such that $\tau_{1}^{m}=\tau_{2}^{n}:=\tau_{0}$ for some $m,n\in\mathbb{N}$, then \eqref{ceq} with $\varepsilon=\varepsilon_{0}$ and $\tau=\tau_{0}$ has two pairs of simple imaginary roots $\pm {\rm i}\omega_{1}$, $\pm {\rm i}\omega_{2}$.
	Moreover, system \eqref{fulleq} may undergo a double Hopf bifurcation when $\Delta(\varepsilon_{0}, \tau_{0}) \neq 0$.
\end{theorem}
\begin{remark}
	\begin{itemize}
		\item[(i)] When $a=1.2$, $\varepsilon=\varepsilon_{0} \approx 0.2231$ and $\tau=\tau_{0} \approx 7.8628$, equation \eqref{ceq} has two pairs of purely imaginary roots $\pm{\rm i}\omega_{1}$ and $\pm{\rm i}\omega_{2}$ with $\omega_{1} \approx 1.0608$, $\omega_{2} \approx 0.9083$, and $\Delta(\varepsilon_{0}, \tau_{0}) \approx -0.2249$.
		\item[(ii)] When $a=1$, $\varepsilon=\varepsilon_{0} \approx 0.2533$ and $\tau=\tau_{0}= 2.5 \pi$, equation \eqref{ceq} has two pairs of purely imaginary roots $\pm{\rm i}\omega_{1}$ and $\pm{\rm i}\omega_{2}$ with $\omega_{1}=1$, $\omega_{2} \approx 0.9674$, and $\Delta(\varepsilon_{0}, \tau_{0}) \approx-1.0070$.
	\end{itemize}
\end{remark}

\subsection{Normal form for the double Hopf bifurcation}\label{sec2.2}
In this subsection, we derive the normal form near the double Hopf bifurcation point $(\varepsilon,\tau)=(\varepsilon_{0},\tau_{0})$ on the center manifold by the method in \cite{L10}.
Letting $x(t)=x_{1}(t)$, $\dot{x}(t)=x_{2}(t)$, system \eqref{fulleq} can be written as
\begin{align}\label{onefulleq}
	\left\{
	\begin{array}{ll}
		\dot{x}_{1}(t)=x_{2}(t),\\
		\dot{x}_{2}(t)=-x_{1}(t)-\varepsilon(x_{1}^{2}(t)-1)x_{2}(t)+\varepsilon f(x_{1}(t-\tau)).
	\end{array}
	\right.
\end{align}
Firstly, we rescale the time by $t \rightarrow \tau t$ and let $\varepsilon=\varepsilon_{0}+\alpha_{1}$, $\tau=\tau_{0}+\alpha_{2}$, $x_{j}(\tau t) \rightarrow x_{j}(t)$, $j=1,2$, then system \eqref{onefulleq} reads
\begin{align}\label{geq}
	\dot{x}(t) = L(\alpha) X_{t}+F(X_{t},\alpha),
\end{align}
where $\alpha=(\alpha_{1},\alpha_{2})$, $x=(x_{1},x_{2})^{T}$, $X_{t} \in {\mathcal C} \equiv C([-1,0],\mathbb{R}^{2})$, $X_{t}(\theta)=x(t+\theta)$ for $-1 \leq \theta \leq 0$, $L(\alpha) : {\mathcal C} \rightarrow \mathbb{R}^{2}$ and $F(\cdot,\alpha) : {\mathcal C} \rightarrow \mathbb{R}^{2}$ are given by
\begin{align*}
	L(\alpha)X_{t}=(\tau_{0}+\alpha_{2}) \left( M(\alpha) X_{t}(0) + N(\alpha) X_{t}(-1) \right) , \quad 
	M(\alpha)=
	\begin{pmatrix}
		0  & 1           \\
		-1 & \varepsilon_{0}+\alpha_{1} \\
	\end{pmatrix},
	\quad
	N(\alpha)=
	\begin{pmatrix}
		0                             & 0 \\
		a(\varepsilon_{0}+\alpha_{1}) & 0 \\
	\end{pmatrix},
\end{align*}
\begin{align*}
	F(X_{t},\alpha)=&(\tau_{0}+\alpha_{2})(\varepsilon_{0}+\alpha_{1})
	\begin{pmatrix}
		0 \\
		-x_{1}^{2}(t)x_{2}(t)+f(x_{1}(t-1))-a x_{1}(t-1)\\
	\end{pmatrix}\\
	=&(\tau_{0}+\alpha_{2})(\varepsilon_{0}+\alpha_{1})
	\begin{pmatrix}
		0 \\
		-x_{1}^{2}(t)x_{2}(t)+b x_{1}^{3}(t-1)+c x_{1}^{4}(t-1)+d x_{1}^{5}(t-1)+{\rm h.o.t.}\\
	\end{pmatrix}.
\end{align*}
According to Theorem \ref{theorem1}, the operator $L(0)$ has two pairs of simple imaginary eigenvalues $\pm{\rm i}\tau_{0}\omega_{1}$ and $\pm{\rm i}\tau_{0}\omega_{2}$ under assumption \eqref{H0}.
Considering the linear equation of \eqref{geq} at the bifurcation point $\alpha=0$
\begin{align*}
	\dot{x}(t) = L(0) X_{t},
\end{align*}
its solution operator is a strongly continuous semigroup on ${\mathcal C}$ with the infinitesimal generator $U_{0}$ given by
\begin{align*}
	U_{0} \phi = \frac{d \phi}{d \theta}, \quad
	{\mathscr D}(U_{0})=\left\lbrace \phi \in {\mathcal C} : \frac{d \phi}{d \theta} \in {\mathcal C},
	\left. \frac{d \phi}{d \theta}\right|_{\theta=0}=L(0) \phi \right\rbrace
\end{align*}
and
\begin{align*}
	\Lambda_{0}:=\left\lbrace \pm {\rm i} \tau_{0} \omega_{1}, \pm {\rm i} \tau_{0} \omega_{2} \right\rbrace \subset \sigma(U_{0}).
\end{align*}
Let ${\mathcal C}^{\textasteriskcentered}:=C([0,1],\mathbb{R}^{2^{\textasteriskcentered}})$, where $\mathbb{R}^{2^{\textasteriskcentered}}$ is the 2-dimensional row-vector space.
Define a bilinear form $\langle \cdot,\cdot \rangle$ on ${\mathcal C}^{\textasteriskcentered} \times {\mathcal C}$ by
\begin{align*}
	\langle \psi,\phi \rangle = \psi(0) \phi(0) + \tau_{0} \int_{0}^{1} \psi(s) N(0) \phi(s-1)ds, \quad 
	\psi \in {\mathcal C}^{\textasteriskcentered}, \quad
	\phi \in {\mathcal C},
\end{align*}
and the formal adjoint operator $U_{0}^{\ast}$ of $U_{0}$ by
\begin{align*}
	\left\langle \psi, U_{0} \phi \right\rangle = \left\langle U_{0}^{\ast} \psi, \phi \right\rangle, \quad \phi \in \mathscr{D}(U_{0}), \quad \psi \in \mathscr{D}(U_{0}^{\ast}).
\end{align*}
Let $\Phi(\theta)=(\phi_{1}(\theta),\cdots,\phi_{4}(\theta))$ be the base of the generalized eigenspace ${\mathscr M}_{\Lambda_{0}}(U_{0})$, and $\Psi(s) =(\psi_{1}(s),\cdots,\psi_{4}(s))^{T}$ be the adjoint base satisfying $\langle \psi_i, \phi_j \rangle = \delta_{ij}$.
The phase space ${\mathcal C}$ is decomposed by $\Lambda_{0}$ as ${\mathcal C}=P \oplus Q$, where
\begin{align*}
	P&={\rm span} \left\lbrace \phi_{1},\cdots,\phi_{4} \right\rbrace,\\
	Q&=\left\lbrace \phi \in {\mathcal C}: \langle \psi,\phi \rangle = 0 \text{ for all } \psi \in {\rm span}\{\psi_{1},\cdots,\psi_{4}\}    \right\rbrace.
\end{align*}
Thus, each $\phi \in {\mathcal C}$ can be expressed as
\begin{align*}
	\phi=\Phi \langle \Psi, \phi \rangle +\phi^{Q}, \quad \phi^{Q} \in Q.
\end{align*}
The solution of \eqref{geq} is decomposed as
\begin{align*}
	X_{t}(\theta)=\Phi(\theta) z(t)+v_{t}(\theta), \quad v_{t} \in Q
\end{align*}
where $z=(z_{1},z_{2},z_{3},z_{4})^{T}$, and we can rewrite \eqref{geq} as
\begin{align}\label{deeq}
	\left\{
	\begin{array}{ll}
		\dot{z}=J z + \Psi(0) G(\Phi z + v_{t},\alpha),\\
		\frac{d}{dt} v_{t} = U_{Q} v_{t} + (I-\pi) K_{0} G(\Phi z+v_{t},\alpha),
	\end{array}\right.
\end{align}
where $J$ satisfies $U_{0} \Phi = \Phi J$ and $U_{Q} = \left. U_{0} \right|_{Q}$,
\begin{align*}
	G(X_{t},\alpha)=(L(\alpha)-L(0))X_{t}+F(X_{t},\alpha),
\end{align*}
\begin{align*}
	(I-\pi)K_{0}(\theta)=\left\{
	\begin{array}{ll}
		- \Phi(\theta) \Psi(0), & \theta \in \left[ -1,0 \right),\\
		E - \Phi(0) \Psi(0), & \theta=0,
	\end{array}
	\right.
\qquad
	K_{0}(\theta)=\left\{
	\begin{array}{ll}
		0, & \theta \in \left[ -1,0 \right),\\
		E, & \theta=0,
	\end{array}
	\right.
\end{align*}
and $E$ is the $2 \times 2$ identity matrix.
If we take $\phi_{1}$ and $\phi_{3}$ such that
\begin{align*}
	U_{0} \phi_{1} = {\rm i} \tau_{0} \omega_{1} \phi_{1}, \quad 
	U_{0} \phi_{3} = {\rm i} \tau_{0} \omega_{2} \phi_{3}
\end{align*}
and $\phi_{2}=\overline{\phi_{1}}$, $\phi_{4}=\overline{\phi_{3}}$, then $z_{2}=\overline{z_{1}}$, $z_{4}=\overline{z_{3}}$ and
\begin{align*}
	J={\rm diag} \left\lbrace  {\rm i} \tau_{0} \omega_{1}, -{\rm i}\tau_{0}\omega_{1}, 
	{\rm i} \tau_{0} \omega_{2}, -{\rm i} \tau_{0} \omega_{2} \right\rbrace.
\end{align*}
For example, we may take
\begin{align*}
	&\phi_{1}=(1,{\rm i} \omega_{1})^{T}e^{{\rm i} \tau_{0} \omega_{1} \theta}, \qquad
	\psi_{1}=D_{1} (-\varepsilon_{0} + {\rm i} \omega_{1},1) e^{-{\rm i} \tau_{0} \omega_{1} s},\\
	&\phi_{3}=(1,{\rm i} \omega_{2})^{T}e^{{\rm i} \tau_{0} \omega_{2} \theta}, \qquad
	\psi_{3}=D_{2} (-\varepsilon_{0} + {\rm i} \omega_{2},1) e^{-{\rm i} \tau_{0} \omega_{2} s},\\
	&D_{j}=(\tau_{0}-\varepsilon_{0}-\tau_{0}\omega_{j}^{2}+{\rm i}\omega_{j}(2-\varepsilon_{0}\tau_{0}))^{-1}, \quad j=1,2,
\end{align*}
which are desired.
\par
The restriction of system \eqref{deeq} on the center manifold is given by 
\begin{align}\label{cen-mani-eq}
	\dot{z}=J z + g(z,\alpha)
\end{align}
for sufficiently small $\left\| \alpha \right\| $, where $g(z,\alpha)=\Psi(0)G(\Phi z + W(z,\theta),\alpha)$ and the center manifold $v_{t}=W(z,\theta)$ satisfies that
\begin{align}\label{cen-mani-ode}
	D_{z} W(z,\theta) J z - U_{Q} W(z,\theta) = (K_{0}-\Phi(\theta)\Psi(0)-D_{z} W(z,\theta)\Psi(0)) G(\Phi z +W(z,\theta),\alpha).
\end{align}
Some exact expressions of $W(z,\theta)$ and $g(z,\alpha)$ are listed in Appendix A.2.
\par
As
\begin{align*}
	k \omega_{1} - l \omega_{2} \neq 0, \quad k,l\in {\mathbb Z}^{+}, \quad
	1 \leq k+l \leq 6,
\end{align*}
it follows from the normal form method that there is an invertible parameter-dependent change of complex coordinate $z=\hat{z}+\sum_{2 \leq j+k+l+m \leq 5}\hat{z}_{1}^{j}\hat{z}_{2}^{k}\hat{z}_{3}^{l}\hat{z}_{4}^{m}$, such that we can normalize system \eqref{cen-mani-eq} for all sufficiently small $\left\| \alpha \right\| $ up to the fifth order terms.
For simplicity of notation, we will drop the hats, and the normalized system \eqref{cen-mani-eq} reads
\begin{align}\label{nf}
	\left\{
	\begin{array}{ll}
		\dot{z}_{1}=&\left( \xi_{1}(\alpha)+{\rm i} \tau_{0} \omega_{1} \right)  z_{1} + a_{11}(\alpha) z_{1}^{2} \bar{z}_{1} + a_{12}(\alpha) z_{1} z_{3} \bar{z}_{3}\\
		& + A_{11}(\alpha) z_{1}^{3} \bar{z}_{1}^{2} + A_{12}(\alpha) z_{1}^{2} \bar{z}_{1} z_{3} \bar{z}_{3} + A_{13}(\alpha) z_{1} z_{3}^{2} \bar{z}_{3}^{2} + O(\left\| z \right\| ^{6}),\\
		\dot{z}_{3}=&\left( \xi_{2}(\alpha)+{\rm i} \tau_{0} \omega_{2} \right) z_{3} + a_{21}(\alpha) z_{1} \bar{z}_{1} z_{3} + a_{22}(\alpha) z_{3}^{2} \bar{z}_{3}\\
		& + A_{21}(\alpha) z_{1}^{2} \bar{z}_{1}^{2} z_{3} + A_{22}(\alpha)z_{1} \bar{z}_{1} z_{3}^{2} \bar{z}_{3} + A_{23}(\alpha) z_{3}^{3} \bar{z}_{3}^{2} + O(\left\| z \right\| ^{6}),
	\end{array}
	\right.
\end{align}
where 
\begin{align*}
	&\xi_{1}(\alpha)=D_{1} \tau_{0} \left( a e^{-{\rm i} \tau_{0} \omega_{1}} + {\rm i} \omega_{1} \right) \alpha_{1}+D_{1} \left( a \varepsilon_{0} e^{-{\rm i} \tau_{0} \omega_{1}} - \omega_{1}^{2}-1 \right) \alpha_{2}+O(\left\| \alpha \right\| ^{2}),\\
	&\xi_{2}(\alpha)=D_{2} \tau_{0} \left( a e^{-{\rm i} \tau_{0} \omega_{2}} + {\rm i} \omega_{2} \right) \alpha_{1}+D_{2} \left( a \varepsilon_{0} e^{-{\rm i} \tau_{0} \omega_{2}} - \omega_{2}^{2}-1 \right) \alpha_{2}+O(\left\| \alpha \right\| ^{2}),\\
	&a_{11}(0)=D_{1} \tau_{0} \varepsilon_{0} \left( 3b e^{-{\rm i} \tau_{0} \omega_{1}}-{\rm i} \omega _{1} \right), \qquad a_{12}(0)=2a_{11}(0),\\
	&a_{22}(0)=D_{2} \tau_{0} \varepsilon_{0} \left( 3b e^{-{\rm i} \tau_{0} \omega_{2}}-{\rm i} \omega _{2} \right), \qquad a_{21}(0)=2a_{22}(0)
\end{align*}
and the calculation of $A_{jl}$ may be found in bifurcation textbooks and papers, see, e.g., \cite{L11}.
We give the value of $A_{jl}$ in Appendix A.3 to discuss the existence of quasi-periodic invariant 3-tori.
By writing $z$ in polar coordinates $(r,\varphi)$, system \eqref{nf} becomes 
\begin{align}\label{r-nf}
	\left\{
	\begin{array}{ll}
		\dot{r}_{1}=r_{1}\left( \delta_{1}(\alpha)+p_{11}(\alpha)r_{1}^{2}+p_{12}(\alpha)r_{2}^{2}
		+P_{11}(\alpha)r_{1}^{4}+P_{12}(\alpha)r_{1}^{2}r_{2}^{2}+P_{13}(\alpha)r_{2}^{4} \right)
		+g_{01}(r,\varphi,\alpha),\\
		\dot{r}_{2}=r_{2}\left( \delta_{2}(\alpha)+p_{21}(\alpha)r_{1}^{2}+p_{22}(\alpha)r_{2}^{2}
		+P_{21}(\alpha)r_{1}^{4}+P_{22}(\alpha)r_{1}^{2}r_{2}^{2}+P_{23}(\alpha)r_{2}^{4} \right)
		+g_{02}(r,\varphi,\alpha),\\
		\dot{\varphi}_{1}=\beta_1(\alpha)+q_{11}(\alpha)r_{1}^{2}+q_{12}(\alpha)r_{2}^{2}
		+Q_{11}(\alpha)r_{1}^{4}+Q_{12}(\alpha)r_{1}^{2}r_{2}^{2}+Q_{13}(\alpha)r_{2}^{4}
		+\frac{1}{r_{1}}g_{03}(r,\varphi,\alpha),\\
		\dot{\varphi}_{2}=\beta_2(\alpha)+q_{21}(\alpha)r_{1}^{2}+q_{22}(\alpha)r_{2}^{2}
		+Q_{21}(\alpha)r_{1}^{4}+Q_{22}(\alpha)r_{1}^{2}r_{2}^{2}+Q_{23}(\alpha)r_{2}^{4}
		+\frac{1}{r_{2}}g_{04}(r,\varphi,\alpha),
	\end{array}
	\right.
\end{align}
where $\delta_{j}={\rm Re}\left\lbrace \xi_{j}\right\rbrace$, $\beta_{j}={\rm Im}\left\lbrace \xi_{j}\right\rbrace+\tau_{0} \omega_{j}$, $p_{jk}={\rm Re}\left\lbrace a_{jk}\right\rbrace$, $q_{jk}={\rm Im}\left\lbrace a_{jk}\right\rbrace$, $P_{jl}={\rm Re}\left\lbrace A_{jl}\right\rbrace$, $Q_{jl}={\rm Im}\left\lbrace A_{jl}\right\rbrace$, $j,k=1,2$, $l=1,2,3$, and $g_{0i}=O(\left\| r \right\| ^{6})$, $i=1,2,3,4$ are $2\pi$-periodic in $\varphi$.
\par
If
\begin{align*}
	\det \left.  \frac{\partial(\delta_{1},\delta_{2})}{\partial(\alpha_{1},\alpha_{2})} \right| _{\alpha=0} =\tau_{0}^{2}\Delta(\varepsilon_{0}, \tau_{0}) \neq 0,
\end{align*}
that is
\begin{align}\label{H1}
	2 a^{2} \tau_{0}+a^2 \tau _0 \varepsilon _0^2-2 a^2 \varepsilon _0-2 \tau _0+\varepsilon _0 \neq 0
\end{align}
holds for system \eqref{r-nf}, then the map $(\alpha_{1},\alpha_{2}) \mapsto (\delta_{1},\delta_{2})$ is regular at $\alpha=0$. 
In other words, we can parametrize a small neighborhood of the origin of the parameter plane $(\alpha_{1},\alpha_{2})$ by $(\delta_{1},\delta_{2})$, and consider $\beta_{j}$, $p_{jk}$, $q_{jk}$, $P_{jl}$, $Q_{jl}$, $g_{0i}$ as the functions in the parameter $\delta=(\delta_{1},\delta_{2})$ in place of $\alpha$.
The previous discussions lead to the following result.
\begin{theorem}\label{theorem2}
	For any fixed $(\varepsilon_{0},\tau_{0})$ with the corresponding eigenvalues $\pm {\rm i} \omega_{1}$, $\pm {\rm i} \omega_{2}$ satisfying \eqref{ceq}, if assumptions \eqref{H0} and \eqref{H1} hold, then system \eqref{fulleq} with $(\varepsilon,\tau)=(\varepsilon_{0}+\alpha_{1},\tau_{0}+\alpha_{2})$ is locally smoothly orbitally equivalent near the origin to 
	\begin{align}\label{r-nf-done}
		\left\{
		\begin{array}{ll}
			\dot{r}_{1}=r_{1}\left( \delta_{1}+p_{11}(\delta)r_{1}^{2}+p_{12}(\delta)r_{2}^{2}
			+P_{11}(\delta)r_{1}^{4}+P_{12}(\delta)r_{1}^{2}r_{2}^{2}+P_{13}(\delta)r_{2}^{4} \right)
			+g_{01}(r,\varphi,\delta),\\
			\dot{r}_{2}=r_{2}\left( \delta_{2}+p_{21}(\delta)r_{1}^{2}+p_{22}(\delta)r_{2}^{2}
			+P_{21}(\delta)r_{1}^{4}+P_{22}(\delta)r_{1}^{2}r_{2}^{2}+P_{23}(\delta)r_{2}^{4} \right)
			+g_{02}(r,\varphi,\delta),\\
			\dot{\varphi}_{1}=\beta_1(\delta)+q_{11}(\delta)r_{1}^{2}+q_{12}(\delta)r_{2}^{2}
			+Q_{11}(\delta)r_{1}^{4}+Q_{12}(\delta)r_{1}^{2}r_{2}^{2}+Q_{13}(\delta)r_{2}^{4}
			+\frac{1}{r_{1}}g_{03}(r,\varphi,\delta),\\
			\dot{\varphi}_{2}=\beta_2(\delta)+q_{21}(\delta)r_{1}^{2}+q_{22}(\delta)r_{2}^{2}
			+Q_{21}(\delta)r_{1}^{4}+Q_{22}(\delta)r_{1}^{2}r_{2}^{2}+Q_{23}(\delta)r_{2}^{4}
			+\frac{1}{r_{2}}g_{04}(r,\varphi,\delta),
		\end{array}
		\right.
	\end{align}
	where 
	\begin{align*}
		\beta_{j}(\delta):=\beta_{j}(\alpha(\delta))=\tau_{0}\omega_{j}+\beta_{j1}\delta_{1}+\beta_{j2}\delta_{2}+O(\left\| \delta \right\| ^{2}),
	\end{align*}
	$p_{jk}(\delta):=p_{jk}(\alpha(\delta))$, $q_{jk}(\delta):=q_{jk}(\alpha(\delta))$, $P_{jl}(\delta):=P_{jl}(\alpha(\delta))$, $Q_{jl}(\delta):=Q_{jl}(\alpha(\delta))$, $j,k=1,2$, $l=1,2,3$,
	$g_{0i}=O(\left\| r \right\| ^{6})$, $i=1,2,3,4$,
	\begin{equation}\label{ex}
		\begin{aligned}
			\left(
			\begin{array}{cc}
				\beta_{11} & \beta_{12} \\
				\beta_{21} & \beta_{22} \\
			\end{array}
			\right)
			=&\left(
			\begin{array}{cc}
				\frac{\tau_{0}  \omega _1 \left(\omega _1^2-1\right) (2-\tau_{0}  \varepsilon_{0} )}{\varepsilon_{0}  \left(\omega _1^2 (2-\tau_{0}  \varepsilon_{0} )^2+\left(-\tau_{0}  \omega _1^2+\tau_{0} -\varepsilon_{0} \right){}^2\right)} & \frac{\omega _1 \left(4 \omega _1^2+\varepsilon_{0} ^2-\tau_{0}  \omega _1^2 \varepsilon_{0} -\tau_{0}  \varepsilon_{0} \right)}{\omega _1^2 (2-\tau_{0}  \varepsilon_{0} )^2+\left(-\tau_{0}  \omega _1^2+\tau_{0} -\varepsilon_{0} \right){}^2} \\
				\frac{\tau_{0}  \omega _2 \left(\omega _2^2-1\right) (2-\tau_{0}  \varepsilon_{0} )}{\varepsilon_{0}  \left(\omega _2^2 (2-\tau_{0}  \varepsilon_{0} )^2+\left(-\tau_{0}  \omega _2^2+\tau_{0} -\varepsilon_{0} \right){}^2\right)} & \frac{\omega _2 \left(4 \omega _2^2+\varepsilon_{0} ^2-\tau_{0}  \omega _2^2 \varepsilon_{0} -\tau_{0}  \varepsilon_{0} \right)}{\omega _2^2 (2-\tau_{0}  \varepsilon_{0} )^2+\left(-\tau_{0}  \omega _2^2+\tau_{0} -\varepsilon_{0} \right){}^2} \\
			\end{array}
			\right)
			\\
			&\times \left(
			\begin{array}{cc}
				\frac{\tau_{0}  \left(\omega _1^2-1\right) \left(\tau_{0}  \omega _1^2-\tau_{0} +\varepsilon_{0} \right)}{\varepsilon_{0}  \left(\omega _1^2 (2-\tau_{0}  \varepsilon_{0} )^2+\left(-\tau_{0}  \omega _1^2+\tau_{0} -\varepsilon_{0} \right){}^2\right)} & \frac{\tau_{0}  \omega _1^2 \left(2 \omega _1^2+\varepsilon_{0} ^2-2\right)}{\omega _1^2 (2-\tau_{0}  \varepsilon_{0} )^2+\left(-\tau_{0}  \omega _1^2+\tau_{0} -\varepsilon_{0} \right){}^2} \\
				\frac{\tau_{0}  \left(\omega _2^2-1\right) \left(\tau_{0}  \omega _2^2-\tau_{0} +\varepsilon_{0} \right)}{\varepsilon_{0}  \left(\omega _2^2 (2-\tau_{0}  \varepsilon_{0} )^2+\left(-\tau_{0}  \omega _2^2+\tau_{0} -\varepsilon_{0} \right){}^2\right)} & \frac{\tau_{0}  \omega _2^2 \left(2 \omega _2^2+\varepsilon_{0} ^2-2\right)}{\omega _2^2 (2-\tau_{0}  \varepsilon_{0} )^2+\left(-\tau_{0}  \omega _2^2+\tau_{0} -\varepsilon_{0} \right){}^2} \\
			\end{array}
			\right)^{-1},\\
			p_{jj}(0)=&\frac{\tau_{0}  \varepsilon_{0}  \left(3 a^2 b \tau_{0}  \varepsilon_{0} -2 a \omega _j^2+a \tau_{0}  \omega _j^2 \varepsilon_{0} -3 b \omega _j^2-3 b\right)}{a \left(\omega _j^2 (2-\tau_{0}  \varepsilon_{0} )^2+\left(-\tau_{0}  \omega _j^2+\tau_{0} -\varepsilon_{0} \right){}^2\right)}, \quad p_{12}(0)=2p_{11}(0), \quad p_{21}(0)=2p_{22}(0),\\
			q_{jj}(0)=&\frac{\tau_{0}  \omega _j \left(a \varepsilon_{0} ^2+a \tau_{0}  \omega _j^2 \varepsilon_{0} -a \tau_{0}  \varepsilon_{0} +6 b \omega _j^2+3 b \varepsilon_{0} ^2-6 b\right)}{a \left(\omega _j^2 (2-\tau_{0}  \varepsilon_{0} )^2+\left(-\tau_{0}  \omega _j^2+\tau_{0} -\varepsilon_{0} \right){}^2\right)}, \quad q_{12}(0)=2q_{11}(0), \quad q_{21}(0)=2q_{22}(0).
		\end{aligned}
	\end{equation}
\end{theorem}

\subsection{Double Hopf bifurcation analysis}\label{sec2.3}
In this subsection, we discuss the existence of quasi-periodic invariant tori for the truncated normal form
\begin{align}\label{t-nf}
	\left\{
	\begin{array}{ll}
		\dot{r}_{1}=r_{1}\left( \delta_{1}+p_{11}(\delta)r_{1}^{2}+p_{12}(\delta)r_{2}^{2}
		+P_{11}(\delta)r_{1}^{4}+P_{12}(\delta)r_{1}^{2}r_{2}^{2}+P_{13}(\delta)r_{2}^{4} \right),\\
		\dot{r}_{2}=r_{2}\left( \delta_{2}+p_{21}(\delta)r_{1}^{2}+p_{22}(\delta)r_{2}^{2}
		+P_{21}(\delta)r_{1}^{4}+P_{22}(\delta)r_{1}^{2}r_{2}^{2}+P_{23}(\delta)r_{2}^{4} \right),\\
		\dot{\varphi}_{1}=\beta_1(\delta),\\
		\dot{\varphi}_{2}=\beta_2(\delta).
	\end{array}
	\right.
\end{align}
Let
\begin{align}\label{T0}
	\rho_{i}=r_{i}^{2}, \quad i=1,2.
\end{align}
Since the first pair of \eqref{t-nf} is independent of the second pair, the existence of invariant tori of \eqref{t-nf} can be studied via the truncated amplitude system
\begin{align}\label{t-a-nf}
	\left\{
	\begin{array}{ll}
		\dot{\rho}_{1}=2 \rho_{1} \left( \delta_{1} +p_{11}(\delta)\rho_{1} +p_{12}(\delta)\rho_{2} +P_{11}(\delta)\rho_{1}^{2} +P_{12}(\delta)\rho_{1}\rho_{2} +P_{13}(\delta)\rho_{2}^{2} \right),\\
		\dot{\rho}_{2}=2 \rho_{2} \left( \delta_{2} +p_{21}(\delta)\rho_{1} +p_{22}(\delta)\rho_{2} +P_{21}(\delta)\rho_{1}^{2} +P_{22}(\delta)\rho_{1}\rho_{2} +P_{23}(\delta)\rho_{2}^{2} \right).
	\end{array}
	\right.
\end{align}
Assume 
\begin{align*}
	p_{jk}(0) \neq 0, \quad j,k=1,2, \qquad \det \left( p_{jk}(0) \right) \neq 0,
\end{align*}
that is
\begin{align}\label{H3}
	3a^{2}b\varepsilon_{0}\tau_{0}-3b+(a\varepsilon_{0}\tau_{0}-2a-3b)\omega_{1,2}^{2} \neq 0.
\end{align}
If assumption \eqref{H3} holds, it follows from \cite{L11} that there are two bifurcation cases depending on whether $p_{11}(0)$ and $p_{22}(0)$ have the same or opposite signs: simple case $p_{11}(0)p_{22}(0)>0$ and difficult case $p_{11}(0)p_{22}(0)<0$.

\par
For the simple case $p_{11}(0)p_{22}(0)>0$, where $p_{11}(0)$ and $p_{22}(0)$ are given in \eqref{ex}, we may consider without loss of generality that $p_{11}(0)<0$ and $p_{22}(0)<0$ (otherwise, reverse time in \eqref{t-a-nf}).
Introducing new phase variables in \eqref{t-a-nf} according to
\begin{align}\label{T1-2}
	\eta_{1}=-p_{11} \rho_{1}, \qquad \eta_{2}=-p_{22} \rho_{2},
\end{align}
we obtain 
\begin{align}\label{t-a-nf-1}
	\left\{
	\begin{array}{ll}
		\dot{\eta}_{1}=2 \eta_{1}\left( \delta_{1} -\eta_{1} -\gamma \eta_{2} +\Gamma_{1} \eta_{1}^{2} +\Gamma_{2} \eta_{1} \eta_{2} +\Gamma_{3} \eta_{2}^{2} \right),\\
		\dot{\eta}_{2}=2 \eta_{2}\left( \delta_{2} -\sigma \eta_{1} -\eta_{2} +\Sigma_{1} \eta_{1}^{2} +\Sigma_{2} \eta_{1} \eta_{2} +\Sigma_{3} \eta_{2}^{2} \right),
	\end{array}
	\right.
\end{align}
where
\begin{equation}\label{sign-1}
	\begin{aligned}
		\gamma=\frac{p_{12}}{p_{22}}, \qquad \Gamma_{1}=\frac{P_{11}}{p_{11}^{2}}, \qquad \Gamma_{2}=\frac{P_{12}}{p_{11}p_{22}}, \qquad \Gamma_{3}=\frac{P_{13}}{p_{22}^{2}},\\
		\sigma=\frac{p_{21}}{p_{11}}, \qquad \Sigma_{1}=\frac{P_{21}}{p_{11}^{2}}, \qquad \Sigma_{2}=\frac{P_{22}}{p_{11}p_{22}}, \qquad \Sigma_{3}=\frac{P_{23}}{p_{22}^{2}},
	\end{aligned}
\end{equation}
where $p_{ij}$, $P_{ik}$ are given in \eqref{ex}.
Using the fact that $p_{12}(0)=2p_{11}(0)$, $p_{21}(0)=2p_{22}(0)$, we have $\gamma(0)\sigma(0)=4$.
System \eqref{t-a-nf-1} has a trivial equilibrium $E_{0}=(0,0)^{T}$ for all $(\delta_{1},\delta_{2})$ and two trivial equilibria $E_{1}=(\delta_{1}+O(\left\| \delta \right\|^{2}),0)^{T}$ and $E_{2}=(0,\delta_{2}+O(\left\| \delta \right\|^{2}))^{T}$ bifurcating from the origin at the bifurcation curves $H_{1}=\left\lbrace (\delta_{1},\delta_{2}): \delta_{1}=O(\left\| \delta \right\|^{2}) \right\rbrace $ and $H_{2}=\left\lbrace (\delta_{1},\delta_{2}): \delta_{2}=O(\left\| \delta \right\|^{2}) \right\rbrace $, respectively. 
In addition to these three trivial equilibria, system \eqref{t-a-nf-1} also has a nontrivial equilibrium
\begin{align*}
	E_{3}=\left( -\frac{\delta_{1}-\gamma \delta_{2}}{\gamma \sigma -1}+O(\left\| \delta \right\|^{2}),\frac{\sigma \delta_{1}-\delta_{2}}{\gamma \sigma -1}+O(\left\| \delta \right\|^{2}) \right)^{T},
\end{align*}
disappearing from the positive quadrant on the bifurcation curves
\begin{align*}
	T_{1}=\left\lbrace (\delta_{1},\delta_{2}): \delta_{1}=\gamma \delta_{2}+O(\delta_{2}^{2}), \quad \delta_{2}>0 \right\rbrace 
\end{align*}
and
\begin{align*}
	T_{2}=\left\lbrace (\delta_{1},\delta_{2}): \delta_{2}=\sigma \delta_{1}+O(\delta_{1}^{2}), \quad \delta_{1}>0 \right\rbrace.
\end{align*}
The equilibrium $E_{3}$ corresponds to an invariant 2-torus of \eqref{t-nf}, and its stability is determined by the two eigenvalues of the following characteristic equation
\begin{align*}
	\lambda^{2}+2(E_{31}+E_{32})\lambda+4E_{31}E_{32}(1-\gamma\sigma)=0,
\end{align*}
where $E_{3i}$ is the $i$th component of $E_{3}$ $(i=1,2)$.
Since $\gamma\sigma(0)=4$, $E_{3}$ is a saddle.
\begin{remark}
	Let $a=1.2$, $b=-6$, $c=d=0$.
	When $\varepsilon_{0} \approx 0.2231$ and $\tau_{0} \approx 7.8628$, we obtain $p_{11}(0)\approx-7.2397$, $p_{22}(0)\approx-13.6504$, $\gamma(0) \approx 1.0607$, $\sigma(0) \approx 3.7710$, $\alpha_{1} \approx 0.1016 \delta_{1}+0.1252 \delta_{2}$, $\alpha_{2} \approx 0.3708 \delta_{1}-0.2511 \delta_{2}$.
	It is a simple case.
	Moreover, the bifurcation curves can obtained,
	\begin{align*}
		H_{1}&=\left\lbrace (\alpha_{1},\alpha_{2}): \alpha_{2}=-2.0058 \alpha_{1}+O(\alpha_{1}^{2}) \right\rbrace,\\
		H_{2}&=\left\lbrace (\alpha_{1},\alpha_{2}): \alpha_{2}=3.6503 \alpha_{1}+O(\alpha_{1}^{2}) \right\rbrace,\\
		T_{1}&=\left\lbrace (\alpha_{1},\alpha_{2}): \alpha_{2}=0.6108 \alpha_{1}+O(\alpha_{1}^{2}), \alpha_{1}>0 \right\rbrace,\\
		T_{2}&=\left\lbrace (\alpha_{1},\alpha_{2}): \alpha_{2}=-1.0041 \alpha_{1}+O(\alpha_{1}^{2}), \alpha_{1}>0 \right\rbrace.
	\end{align*}
	The parameter plane of $(\alpha_{1},\alpha_{2})$ around the double Hopf bifurcation point is divided into six regions as shown in Figure \ref{2-fzt}.
	When $(\alpha_{1},\alpha_{2})$ lies in region 1 in Figure \ref{2-fzt}, the zero solution of system \eqref{t-nf} is asymptotically stable.
	When $(\alpha_{1},\alpha_{2})$ lies in regions 2 and 5, the zero solution of system \eqref{t-nf} loses its stability and a stable periodic solution emerges.
	When $(\alpha_{1},\alpha_{2})$ lies in regions 3 and 6, the zero solution of system \eqref{t-nf} is also unstable, the stable and unstable periodic solutions coexist. 
	When $(\alpha_{1},\alpha_{2})$ lies in region 4, system \eqref{t-nf} has two stable periodic solutions and an unstable quasi-periodic solution.
\end{remark}
\begin{figure}
	\centering
	\includegraphics[width=150mm, height=60mm]{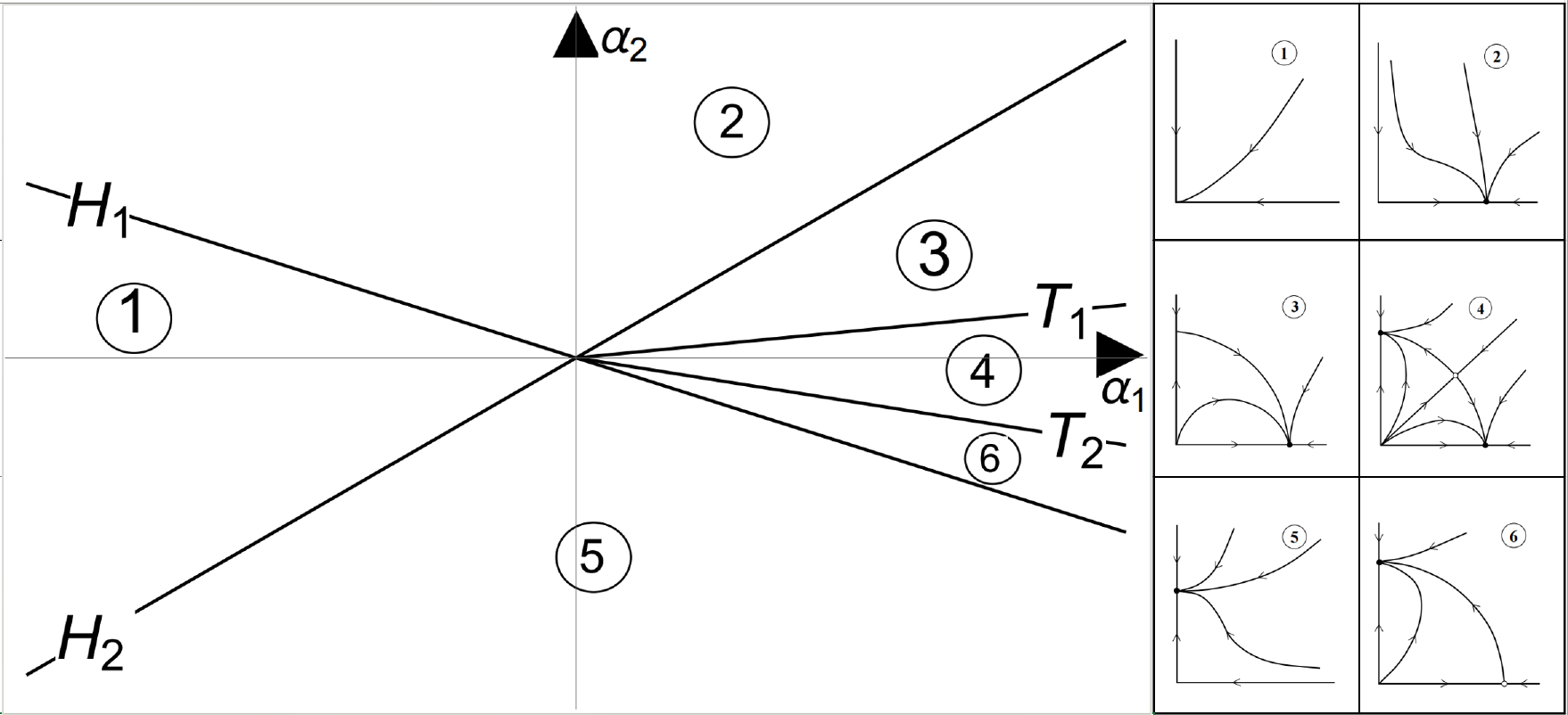}
	\caption{The bifurcation diagram and phase portrait of system \eqref{fulleq} with critical value $(\varepsilon_{0},\tau_{0}) \approx (0.2231,7.8628)$, where $a=1.2$, $b=-6$, $c=d=0$.}
	\label{2-fzt}
\end{figure}

\par
On the other hand, when
\begin{align}\label{HH0}
	p_{11}(0)p_{22}(0)<0,
\end{align}
similarly to the previous case, consider only $p_{11}(0)>0$, $p_{22}(0)<0$ (otherwise, exchange the subscripts in \eqref{t-a-nf}).
Introducing new phase variables
\begin{align}\label{T1-3}
	\eta_{1}=p_{11} \rho_{1}, \qquad \eta_{2}=-p_{22} \rho_{2},
\end{align}
the truncated amplitude system \eqref{t-a-nf} is written as
\begin{align}\label{t-a-nf-2}
	\left\{
	\begin{array}{ll}
		\dot{\eta}_{1}=2 \eta_{1}\left( \delta_{1} +\eta_{1} -\gamma \eta_{2} +\Gamma_{1} \eta_{1}^{2} -\Gamma_{2} \eta_{1} \eta_{2} +\Gamma_{3} \eta_{2}^{2} \right),\\
		\dot{\eta}_{2}=2 \eta_{2}\left( \delta_{2} +\sigma \eta_{1} -\eta_{2} +\Sigma_{1} \eta_{1}^{2} -\Sigma_{2} \eta_{1} \eta_{2} +\Sigma_{3} \eta_{2}^{2} \right),
	\end{array}
	\right.
\end{align}
where $\gamma$, $\sigma$, $\Gamma_{i}$, $\Sigma_{i}$ $(i=1,2,3)$ are the same as before.
System \eqref{t-a-nf-2} has three trivial equilibria $E_{0}=(0,0)^{T}$, $E_{1}=(-\delta_{1}+O(\left\| \delta \right\|^{2}),0)^{T}$, $E_{2}=(0,\delta_{2}+O(\left\| \delta \right\|^{2}))^{T}$ and a nontrivial equilibrium 
\begin{align*}
	E_{3}=\left( \frac{\delta_{1}-\gamma \delta_{2}}{\gamma \sigma -1}+O(\left\| \delta \right\|^{2}),\frac{\sigma \delta_{1}-\delta_{2}}{\gamma \sigma -1}+O(\left\| \delta \right\|^{2}) \right)^{T},
\end{align*}
whenever their coordinates are non-negative.
The characteristic equation of \eqref{t-a-nf-2} about the fixed point $E_{3}$ is
\begin{align*}
	\lambda^{2}+2(E_{32}-E_{31})\lambda+4E_{31}E_{32}(\gamma\sigma-1)=0,
\end{align*}
where $E_{3i}$ is the $i$th component of $E_{3}$ $(i=1,2)$.
Simple analysis produces that $E_{3}$ is a sink if $(1-\gamma)\delta_{2}<(\sigma-1)\delta_{1}$ and a source if $(1-\gamma)\delta_{2}>(\sigma-1)\delta_{1}$.
\par
Moreover, system \eqref{t-a-nf-2} may have a limit cycle arising from the Hopf bifurcation of $E_{3}$ and the Hopf bifurcation curve is
\begin{align*}
	H=\left\lbrace (\delta_{1},\delta_{2}): \delta_{2}=-\frac{\sigma-1}{\gamma-1}\delta_{1}+O(\delta_{1}^{2}), \quad \delta_{1}-\gamma \delta_{2}>0, \quad \sigma \delta_{1}-\delta_{2}>0 \right\rbrace 
\end{align*}
if
\begin{align}\label{HH78}
	\gamma(0) \neq 1, \qquad \sigma(0) \neq 1.
\end{align}
The value of the first Lyapunov coefficient $l_{1}$ along the Hopf curve $H$ near the origin, as the nondegeneracy condition, should be nonzero, that is
\begin{small}
\begin{align}\label{HH9}
	l_{10}:=(\frac{\sigma -1}{\gamma-1})\left( -4(\Gamma_{1}+\Gamma_{2}+\Sigma_{2}+\Sigma_{3}) +(7\Gamma_{1}+\Sigma_{1}+\Sigma_{2}-\gamma\Sigma_{1})\gamma +(\Gamma_{2}+\Gamma_{3}+7\Sigma_{3}-\sigma\Gamma_{3})\sigma \right)(0) \neq 0,
\end{align}
\end{small}
where $\gamma(0)$, $\sigma(0)$, $\Gamma_{i}(0)$, $\Sigma_{i}(0)$, $i=1,2,3$, in \eqref{HH78} and \eqref{HH9} are given in \eqref{sign-1}.
The stability of the limit cycle of system \eqref{t-a-nf-2} is determined by the sign of the first Lyapunov coefficient $l_{10}$.
If $l_{1}(0)<0$, $E_{3}$ is a sink when $(1-\gamma)\delta_{2}<(\sigma-1)\delta_{1}$;
$E_{3}$ is a source and a stable limit cycle bifurcate from $E_{3}$ when $(1-\gamma)\delta_{2}>(\sigma-1)\delta_{1}$.
If $l_{1}(0)>0$, $E_{3}$ is a source when $(1-\gamma)\delta_{2}>(\sigma-1)\delta_{1}$;
$E_{3}$ is a sink and a unstable limit cycle bifurcate from $E_{3}$ when $(1-\gamma)\delta_{2}<(\sigma-1)\delta_{1}$.
The equilibrium $E_{3}$ and the limit cycle of \eqref{t-a-nf-2} correspond to an invariant 2-torus and an invariant 3-torus of \eqref{t-nf}, respectively.
\begin{remark}
\begin{itemize}
	\item[(i)]Let $a=1$, $b=\frac{1}{7}$, $c=d=0$.
	When $\varepsilon_{0} \approx 0.2533$ and $\tau_{0}=2.5\pi$, we obtain $p_{11}(0)\approx -0.4837$, $p_{22}(0)\approx 0.3976$, $\gamma(0) \approx -2.4333$, $\sigma(0) \approx -1.6439$, $\alpha_{1} \approx 0.1207 \delta_{1}+0.1262 \delta_{2}$, $\alpha_{2} \approx 0.1276 \delta_{1}$, $l_{10}\approx-6591.28$.
	It is a difficult case, and the limit cycle bifurcating from $E_{3}$ is stable.
	Moreover, the bifurcation curves can obtained
	\begin{align*}
		H_{1}&=\left\lbrace (\alpha_{1},\alpha_{2}): \alpha_{2}=0+O(\alpha_{1}^{2}) \right\rbrace,\\
		H_{2}&=\left\lbrace (\alpha_{1},\alpha_{2}): \alpha_{2}=1.0569 \alpha_{1}+O(\alpha_{1}^{2}) \right\rbrace,\\ 
		T_{1}&=\left\lbrace (\alpha_{1},\alpha_{2}): \alpha_{2}=1.8526 \alpha_{1}+O(\alpha_{1}^{2}), \alpha_{1}<0 \right\rbrace,\\
		T_{2}&=\left\lbrace (\alpha_{1},\alpha_{2}): \alpha_{2}=-1.4716 \alpha_{1}+O(\alpha_{1}^{2}), \alpha_{1}>0 \right\rbrace,\\
		H&=\left\lbrace (\alpha_{1},\alpha_{2}): \alpha_{2}=5.4161 \alpha_{1}+O(\alpha_{1}^{2}), \alpha_{1}<0 \right\rbrace,\\
		Y&=\left\lbrace (\delta_{1},\delta_{2}): \delta_{2}=-0.7701 \delta_{1}-131.18 \delta_{1}^{2}+O(\delta_{1}^{3}),  \delta_{1}-\gamma \delta_{2}>0, \sigma \delta_{1}-\delta_{2}>0 \right\rbrace,
	\end{align*}
	where the curve $Y$ is a heteroclinic bifurcation via which the 3-torus disappears \cite{L11}.
	The parameter plane of $(\alpha_{1},\alpha_{2})$ around the double Hopf bifurcation point is divided into eight regions as shown in Figure \ref{3-fzt}.
	When $(\alpha_{1},\alpha_{2})$ lies in region 2 in Figure \ref{3-fzt}, the zero solution of system \eqref{t-nf} is unstable.
	When $(\alpha_{1},\alpha_{2})$ lies in regions 1 and 3, system \eqref{t-nf} has an unstable periodic solution bifurcating from the origin.
	When $(\alpha_{1},\alpha_{2})$ lies in region 4, system \eqref{t-nf} has two unstable periodic solutions.
	When $(\alpha_{1},\alpha_{2})$ lies in region 8, the stable and unstable periodic solutions coexist.
	When $(\alpha_{1},\alpha_{2})$ lies in region 7, system \eqref{t-nf} has two unstable periodic solutions and a stable quasi-periodic solution.
	When $(\alpha_{1},\alpha_{2})$ lies in region 6, the stable quasi-periodic solution in region 7 loses its stability, and a stable three-dimensional torus emerges.
	When $(\alpha_{1},\alpha_{2})$ lies in region 5, the three-dimensional torus that appears in region 6 disappears.
	\item[(ii)]Let $a=1$, $b=\frac{1}{6}$, $c=d=0$.
	When $\varepsilon_{0} \approx 0.2533$ and $\tau_{0}=2.5\pi$, we obtain $p_{11}(0)\approx -0.507885$, $p_{22}(0)\approx 0.517818$, $\gamma(0) \approx -1.9616$, $\sigma(0) \approx -2.0391$, $l_{10}\approx779.1041$.
	The limit cycle bifurcating from $E_{3}$ is unstable.
\end{itemize}
\end{remark}
\begin{figure}
	\centering
	\includegraphics[width=150mm, height=50mm]{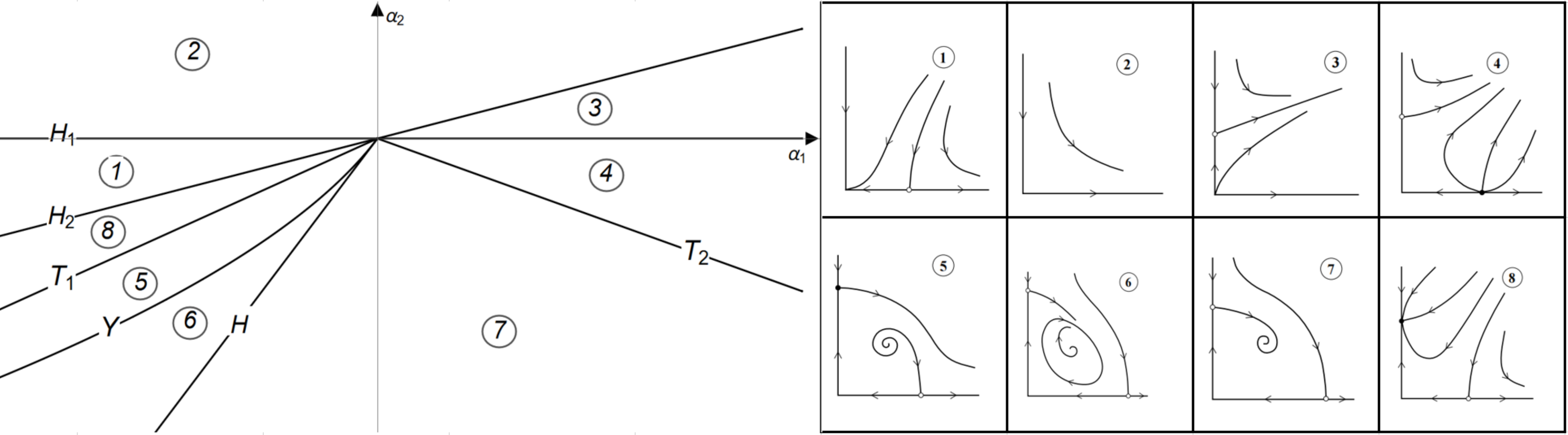}
	\caption{The bifurcation diagram and phase portrait of system \eqref{fulleq} with critical value $\varepsilon_{0} \approx 0.2533$ and $\tau_{0}=2.5\pi$, , where $a=1$, $b=\frac{1}{7}$, $c=d=0$.}
	\label{3-fzt}
\end{figure}
\begin{theorem}\label{theorem2.3}
	Suppose that \eqref{H3} holds for system \eqref{t-nf}. Then the following statements are true.
	\begin{itemize}
		\item[(i)] When $p_{11}(0)p_{22}(0)>0$, system \eqref{t-nf} has an invariant 2-torus for some sufficiently small $\delta$ such that
		\begin{align*}
			E_{3}=\left( -\frac{\delta_{1}-\gamma \delta_{2}}{\gamma \sigma -1}+O(\left\| \delta \right\|^{2}),\frac{\sigma \delta_{1}-\delta_{2}}{\gamma \sigma -1}+O(\left\| \delta \right\|^{2}) \right)^{T}
		\end{align*}
		is in the positive quadrant.
		\item[(ii)] When $p_{11}(0)p_{22}(0)<0$, system \eqref{t-nf} has an invariant 2-torus for some sufficiently small $\delta$ such that
		\begin{align*}
			E_{3}=\left( \frac{\delta_{1}-\gamma \delta_{2}}{\gamma \sigma -1}+O(\left\| \delta \right\|^{2}),\frac{\sigma \delta_{1}-\delta_{2}}{\gamma \sigma -1}+O(\left\| \delta \right\|^{2}) \right)^{T}
		\end{align*}
		is in the positive quadrant.
		Furthermore, if assumptions \eqref{HH78}, \eqref{HH9} hold and $l_{1}(0)<0(>0)$, a stable (unstable) invariant 3-torus of \eqref{t-nf} appears as $(\delta_{1},\delta_{2})$ crosses the Hopf curve $H$ near the origin.
	\end{itemize}
\end{theorem}

\section{Quasi-periodic invariant tori}\label{sec3}
\subsection{Existence of quasi-periodic invariant 2-tori}\label{sec3.1}
Form the discussion in the last section, under assumption \eqref{H3}, the truncated system \eqref{t-nf} has invariant 2-tori for some parameters $\delta$ such that equilibrium $E_{3}$ is in the first quadrant.
In this subsection, we consider the persistence of invariant 2-tori of the truncated system \eqref{t-nf} when the higher-order terms are taken into account.
Without loss of generality, in the following, we discuss the case that $p_{11}(0)<0$ and $p_{22}(0)<0$.
For the case that $p_{11}(0)p_{22}(0)<0$, we have the similar result.
In order to obtain the order of perturbations, we introduce a small parameter $\epsilon$, and rescale $\eta_{i}$ and $\delta_{i}$ by $\epsilon \eta_{i}$ and $\epsilon \delta_{i}$, $i=1,2$.
System \eqref{t-a-nf-1} becomes
\begin{align}\label{taeq}
	\left\{
\begin{array}{ll}
	\dot{\eta}_{1}=2 \epsilon \eta_{1}\left( \delta_{1} -\eta_{1} -\gamma \eta_{2} +\epsilon\Gamma_{1} \eta_{1}^{2} +\epsilon\Gamma_{2} \eta_{1} \eta_{2} +\epsilon\Gamma_{3} \eta_{2}^{2} \right),\\
	\dot{\eta}_{2}=2 \epsilon \eta_{2}\left( \delta_{2} -\sigma \eta_{1} -\eta_{2} +\epsilon\Sigma_{1} \eta_{1}^{2} +\epsilon\Sigma_{2} \eta_{1} \eta_{2} +\epsilon\Sigma_{3} \eta_{2}^{2} \right),
\end{array}
\right.
\end{align}
which has a positive equilibrium $\eta=y_{0}=(y_{10},y_{20})^{T}$ for some parameters $\delta$, where
\begin{align*}
	y_{10}=\frac{-\delta_{1}+\gamma\delta_{2}}{\gamma\sigma-1}+O(\epsilon), \qquad
	y_{20}=\frac{\sigma\delta_{1}-\delta_{2}}{\gamma\sigma-1}+O(\epsilon).
\end{align*} 
Let $\Pi_{2}$ denote the set of all such parameters $\delta$.
Substituting \eqref{T0}, \eqref{T1-2} and these rescaling into \eqref{r-nf-done}, and making the translational change $\eta=y_{0}+\epsilon^{\frac{1}{4}}I$, we obtain
\begin{align}\label{kam-nf}
	\left\{
	\begin{array}{ll}
		\dot{I}=\epsilon \Omega^{0}(\delta,\epsilon)I+\epsilon^{\frac{5}{4}}F_{1}(I,\varphi,\delta,\epsilon),\\
		\dot{\varphi}=\omega^0(\delta,\epsilon)+\epsilon^{\frac{5}{4}}F_{2}(I,\varphi,\delta,\epsilon),
	\end{array}
	\right.
\end{align}
where
\begin{align*}
	\Omega^{0}(\delta,\epsilon)=\left(
	\begin{array}{cc} 
		-2y_{10} & -2\gamma y_{10}\\
		-2\sigma y_{20} & -2y_{20}
	\end{array} 
	\right), 
	\qquad
	\omega^0(\delta,\epsilon)=\left(
	\begin{array}{c}
		\beta_{1}(\epsilon \delta) -\epsilon \left( \frac{q_{11}(\epsilon \delta)}{p_{11}(\epsilon \delta)}y_{10} +\frac{q_{12}(\epsilon \delta)}{p_{22}(\epsilon \delta)} y_{20} \right) \\
		\beta_{2}(\epsilon \delta) -\epsilon \left( \frac{q_{21}(\epsilon \delta)}{p_{11}(\epsilon \delta)} y_{10} +\frac{q_{22}(\epsilon \delta)}{p_{22}(\epsilon \delta)}y_{20} \right)
	\end{array}
	\right),
\end{align*}
and $F=(F_{1},F_{2})^{T}$ is the function of $I$, $\varphi$, $\delta$, $\epsilon$.
Take a convex bounded open subset $\Pi_{2}^{0} \subset \Pi_{2}$ with positive Lebesgue measure. 
Let $\lambda_{1}^{0}(\delta,\epsilon)$ and $\lambda_{2}^{0}(\delta,\epsilon)$ be the two eigenvalues of $\Omega^{0}(\delta,\epsilon)$.
Denote $\lambda_{i0}^{0}(\delta)=\lambda_{i}^{0}(\delta,0) (i=1,2)$, $\Lambda_{0}^{0}(\delta)={\rm diag}\{\lambda_{10}^{0}(\delta),\lambda_{20}^{0}(\delta)\}$.
$\omega^{0}$ can be write as $\omega^{0}(\delta,\epsilon)=\omega_{0}^{0}+\epsilon \omega_{1}^{0}(\delta) +O(\epsilon^2)$ with $\omega_{0}^{0}=\tau_{0}(\omega_{1},\omega_{2})^{T}$.
We can use the results in \cite{L12} to obtain the existence of quasi-periodic invariant 2-tori of \eqref{kam-nf}, which implies the existence of quasi-periodic invariant 2-tori of \eqref{fulleq}.
\begin{theorem}\label{theorem3}
	If
	\begin{align}\label{the3-mes}
		\det \left( \frac{\partial \omega_{1}^{0}(\delta)}{\partial \delta}  \right) = \beta_{11}\beta_{22}-\beta_{12}\beta_{21}-\frac{q_{22}(0)}{p_{22}(0)}\beta_{11}-\frac{q_{11}(0)}{p_{11}(0)}\beta_{22}+\frac{q_{11}(0)q_{22}(0)}{p_{11}(0)p_{22}(0)} \neq 0,
	\end{align}
	where $\beta_{ij}$, $p_{ij}(0)$, $q_{ij}(0)$ are given in \eqref{ex}, then system \eqref{kam-nf} has quasi-periodic invariant 2-tori for most parameter values of $\Pi_{2}^{0}$ and sufficiently small $\epsilon$.
\end{theorem}
\begin{proof}
	The theorem can be proved by Theorem 2.2 and 2.3 in \cite{L12}.
	System \eqref{kam-nf} is a special case of (1.1) in \cite{L12} with $n_{11}=n_{21}=0$, $n_{12}=n_{22}=2$, $q_{3}=1$, $q_{4}=\frac{1}{4}$, $q_{7}=\frac{5}{4}$.
	Simple operation produces that $\lambda_{i0}^{0}(\delta)=-(y_{10}+y_{20})\pm \sqrt{(y_{10}-y_{20})^{2}+4\gamma\sigma y_{10}y_{20}}$.
	Since $y_{10}>0$, $y_{20}>0$ for $\delta \in \Pi_{2}^{0}$ and $\gamma(0)\sigma(0)=4$, there is a positive constant $c_{0}$ such that
	\begin{align*}
		\left| \lambda_{i0}^{0}(\delta) \right| \geq c_{0}(i=1,2), \qquad \left| \lambda_{10}^{0}(\delta)-\lambda_{20}^{0}(\delta) \right| \geq c_{0}, \qquad \forall \delta \in \Pi_{2}^{0},
	\end{align*}
	Obviously, hypotheses (H1)-(H3) in \cite{L12} are satisfied.
	Thus it follows from Theorem 2.2 in \cite{L12} that for any given $0< \kappa \ll 1$, there is a sufficiently small $\epsilon^{\star}>0$ such that when $0 < \epsilon < \epsilon^{\star}$, there is a Cantor set $\Pi_{\kappa} \subset \Pi_{2}^{0}$ such that for each $\delta \in \Pi_{\kappa}$, system \eqref{kam-nf} has quasi-periodic invariant 2-tori.
	Moreover, it follows from assumption \eqref{the3-mes} and Theorem 2.3 in \cite{L12} that ${\rm meas} \left( \Pi_{2}^{0} - \Pi_{\kappa} \right) \rightarrow 0 $ as $\kappa \rightarrow 0$, that is, for most parameter values of $\Pi_{2}^{0}$, system \eqref{kam-nf} has quasi-periodic invariant 2-tori.
\end{proof}
Because these coordinate and parameter transformations changing \eqref{fulleq} into \eqref{kam-nf} are reversible, we have the following result.
\begin{corollary}\label{corollary1}
	For any fixed $(\varepsilon_{0},\tau_{0})$ with the corresponding eigenvalues $\pm {\rm i} \omega_{1}$, $\pm {\rm i} \omega_{2}$ satisfying \eqref{ceq}, if assumptions \eqref{H0}, \eqref{H1}, \eqref{H3} and \eqref{the3-mes} hold, then system \eqref{fulleq} has quasi-periodic invariant 2-tori for infinitely many $(\varepsilon,\tau)$ near $(\varepsilon_{0},\tau_{0})$.
\end{corollary}

\subsection{Existence of Quasi-periodic invariant 3-tori}\label{sec3.2}
When $p_{11}(0)p_{22}(0)<0$, the truncated amplitude system \eqref{t-a-nf} may have limit cycles associated with a secondary Hopf bifurcation.
Without loss of generality, in the following, we consider the case that $p_{11}(0)>0$ and $p_{22}(0)<0$.
In this subsection, we derive the normal form near the Hopf bifurcation curve $H$ and discuss the persistence of the quasi-periodic invariant 3-tori corresponding to the limit cycles in \eqref{t-a-nf} by KAM theorem.
In order to obtain a smaller perturbation, we need to normalize \eqref{r-nf-done} to higher order terms.
As
\begin{align}\label{C1}
	k \omega_{1} - l \omega_{2} \neq 0, \quad k,l\in {\mathbb Z}^{+}, \quad
	1 \leq k+l \leq 12,
\end{align}
there is an invertible parameter-dependent change such that \eqref{r-nf-done} can be normalized to the eleventh-order terms, and using this change with
\begin{align*}
	\eta \rightarrow \epsilon \eta, \qquad \delta \rightarrow \epsilon \delta,
\end{align*}
we obtain
\begin{align}\label{2}
	\left\{
	\begin{array}{ll}
		\dot{\eta}_{1}=2 \epsilon \eta_{1}\left( \delta_{1} +\eta_{1} -\gamma\eta_{2} +\epsilon\Gamma_{1} \eta_{1}^{2} -\epsilon\Gamma_{2} \eta_{1} \eta_{2} +\epsilon\Gamma_{3} \eta_{2}^{2}+\sum_{3 \leq l+m \leq 5}\epsilon^{l+m-1}\gamma_{lm}\eta_{1}^{l}\eta_{2}^{m} \right)+\epsilon^{\frac{11}{2}}\eta_{1}^{\frac{1}{2}}f_{11}(\eta,\varphi,\delta),\\
		\dot{\eta}_{2}=2 \epsilon \eta_{2}\left( \delta_{2} +\sigma \eta_{1} -\eta_{2} +\epsilon\Sigma_{1} \eta_{1}^{2} -\epsilon\Sigma_{2} \eta_{1} \eta_{2} +\epsilon\Sigma_{3} \eta_{2}^{2}+\sum_{3 \leq l+m \leq 5}\epsilon^{l+m-1}\sigma_{lm}\eta_{1}^{l}\eta_{2}^{m} \right)+\epsilon^{\frac{11}{2}}\eta_{2}^{\frac{1}{2}}f_{12}(\eta,\varphi,\delta),\\
		\dot{\varphi}_1=\beta_{1}^{(1)} +\epsilon b_{11}\eta_{1}+\epsilon b_{12}\eta_{2}+\sum_{2 \leq l+m \leq 5}\epsilon^{l+m}B_{1lm}\eta_{1}^{l}\eta_{2}^{m}+\epsilon^{\frac{11}{2}}\eta_{1}^{-\frac{1}{2}}f_{13}(\eta,\varphi,\delta),\\
		\dot{\varphi}_2=\beta_{2}^{(1)}  +\epsilon b_{21}\eta_{1}+\epsilon b_{22}\eta_{2}+\sum_{2 \leq l+m \leq 5}\epsilon^{l+m}B_{2lm}\eta_{1}^{l}\eta_{2}^{m}+\epsilon^{\frac{11}{2}}\eta_{2}^{-\frac{1}{2}}f_{14}(\eta,\varphi,\delta),
	\end{array}
	\right.
\end{align}
where $\gamma$, $\sigma$, $\Gamma_{i}$, $\Sigma_{i}$ $(i=1,2,3)$ are the same as before, and $f_{1i}=O(\left\| \eta \right\|^{6} )$, $(i=1,\cdots,4)$,
\begin{align*}
		b_{11}=\frac{q_{11}}{p_{11}}, \qquad b_{12}=\frac{-q_{12}}{p_{22}}, \qquad b_{21}=\frac{q_{21}}{p_{11}}, \qquad b_{22}=\frac{-q_{22}}{p_{22}},
\end{align*}
\begin{align*}
	\beta_{j}^{(1)}(\delta) =\tau_{0}\omega_{j}+\epsilon\beta_{j1}\delta_{1}+\epsilon\beta_{j2}\delta_{2}+\epsilon^{2}X(\delta), \qquad j=1,2.
\end{align*}
It is well-known that the first pair of equations in \eqref{2} without the terms $\epsilon^{\frac{11}{2}}\eta_{i}^{\frac{1}{2}}f_{1i}$, $(i=1,2)$ has a positive equilibrium $\eta_{0}=(\eta_{10},\eta_{20})^{T}$ for some $\delta$, where
\begin{align*}
	\eta_{10}=\frac{\delta_{1}-\gamma\delta_{2}}{\gamma\sigma-1}+O(\epsilon), \qquad
	\eta_{20}=\frac{\sigma\delta_{1}-\delta_{2}}{\gamma\sigma-1}+O(\epsilon).
\end{align*} 
Since the map $(\delta_{1},\delta_{2}) \mapsto (\eta_{10},\eta_{20})$ is a diffeomorphism under assumption \eqref{H3}, we can introduce new parameter variable $\zeta=(\zeta_{1},\zeta_{2})^{T}$ by setting
\begin{align*}
	\zeta_{1}=\eta_{10}, \qquad \zeta_{2}=\eta_{20}.
\end{align*}
By the coordinate change
\begin{align*}
	\eta \rightarrow \eta_{0}+\eta,
\end{align*}	 
\eqref{2} is changed into
\begin{align}\label{3}
	\left\{
	\begin{array}{ll}
		\dot{\eta}_1=2\epsilon \left(  c_{11}\eta_{1}+c_{12}\eta_{2}+C_{11}\eta_{1}^{2}+C_{12}\eta_{1}\eta_{2}+C_{13}\eta_{2}^{2}+\sum_{3 \leq l+m \leq 6}\epsilon^{l+m-2}\gamma_{lm}^{'}\eta_{1}^{l}\eta_{2}^{m} \right)+\epsilon^{\frac{11}{2}}f_{21}(\eta,\varphi,\zeta),\\
		\dot{\eta}_2=2\epsilon \left( c_{21}\eta_{1}+c_{22}\eta_{2}+C_{21}\eta_{1}^{2}+C_{22}\eta_{1}\eta_{2}+C_{23}\eta_{2}^{2}+\sum_{3 \leq l+m \leq 6}\epsilon^{l+m-2}\sigma_{lm}^{'}\eta_{1}^{l}\eta_{2}^{m} \right)+\epsilon^{\frac{11}{2}}f_{22}(\eta,\varphi,\zeta),\\
		\dot{\varphi}_1=\beta_{1}^{(2)} +\epsilon b_{11}^{(2)}\eta_{1}+\epsilon b_{12}^{(2)}\eta_{2}+\sum_{2 \leq l+m \leq 5}\epsilon^{l+m}B_{1lm}^{(2)}\eta_{1}^{l}\eta_{2}^{m}+\epsilon^{\frac{11}{2}}f_{23}(\eta,\varphi,\zeta),\\
		\dot{\varphi}_2=\beta_{2}^{(2)} +\epsilon b_{21}^{(2)}\eta_{1}+\epsilon b_{22}^{(2)}\eta_{2}+\sum_{2 \leq l+m \leq 5}\epsilon^{l+m}B_{2lm}^{(2)}\eta_{1}^{l}\eta_{2}^{m}+\epsilon^{\frac{11}{2}}f_{24}(\eta,\varphi,\zeta),
	\end{array}
	\right.
\end{align}
where $b_{ij}^{(2)}=b_{ij}+O(\epsilon)$,
\begin{align*}
	c_{11}(\zeta,\epsilon)&=\zeta_1(1+\epsilon(2\Gamma_{1}\zeta_1-\Gamma_{2}\zeta_2))+O(\epsilon^{2}), &
	c_{21}(\zeta,\epsilon)&=\zeta_2(\sigma+\epsilon(2\Sigma_{1}\zeta_1-\Sigma_{2}\zeta_2))+O(\epsilon^{2}),\\
	c_{12}(\zeta,\epsilon)&=\zeta_1(-\gamma+\epsilon(-\Gamma_{2}\zeta_1+2\Gamma_{3}\zeta_2))+O(\epsilon^{2}), &
	c_{22}(\zeta,\epsilon)&=\zeta_2(-1+\epsilon(-\Sigma_{2}\zeta_1+2\Sigma_{3}\zeta_2))+O(\epsilon^{2}),\\
	C_{11}(\zeta,\epsilon)&=1+\epsilon(3\Gamma_{1}\zeta_1-\Gamma_{2}\zeta_2)+O(\epsilon^{2}), &
	C_{21}(\zeta,\epsilon)&=\epsilon \Sigma_{1}\zeta_2+O(\epsilon^{2}),\\
	C_{12}(\zeta,\epsilon)&=-\gamma+\epsilon(-2\Gamma_{2}\zeta_1+2\Gamma_{3}\zeta_2)+O(\epsilon^{2}), &
	C_{22}(\zeta,\epsilon)&=\sigma+\epsilon(2\Sigma_{1}\zeta_1-2\Sigma_{2}\zeta_2)+O(\epsilon^{2}),\\
	C_{13}(\zeta,\epsilon)&=\epsilon \Gamma_{3}\zeta_1+O(\epsilon^{2}), &
	C_{23}(\zeta,\epsilon)&=-1+\epsilon(-\Sigma_{2}\zeta_1+3\Sigma_{3}\zeta_2)+O(\epsilon^{2}).
\end{align*}
\begin{align*}
	\beta_{j}^{(2)} =\tau_{0}\omega_{j}+\epsilon(b_{j1}-\beta_{j1}-\sigma\beta_{j2})\zeta_{1}+\epsilon(b_{j2}+\gamma\beta_{j1}+\beta_{j2})\zeta_{2}+\epsilon^{2}X(\zeta), \qquad j=1,2,
\end{align*}
and all coefficients are functions of $\zeta$, $\epsilon$.
Here and below, $X(\cdot)$ denotes some smooth function in ``$\cdot$".
\par
The truncated amplitude system of \eqref{3} will undergo the secondary Hopf bifurcation near the curve
\begin{align*}
	H^{\star}=\{ (\zeta_{1},\zeta_{2})^{T} : c_{11}+c_{22}=0, \quad \zeta_{1}>0, \quad \zeta_{2}>0 \},
\end{align*}
and may have a limit cycle from the origin $\eta=0$.
The curve is sufficiently close to the Hopf bifurcation curve H.
We regard $\alpha_{3}=\frac{1}{2}(c_{11}+c_{22})$ as the Hopf bifurcation parameter, and obtain
\begin{align*}
	\zeta_{1}=&2\alpha_{3}+\zeta_2+\epsilon \left( (-2 \Sigma_{3}+\Gamma_{2}+\Sigma_{2}-2\Gamma_{1})\zeta_2^2+2 \left(-4 \Gamma_{1} +\Gamma_{2}+\Sigma_{2}\right)\zeta_2 \alpha_{3}-8 \Gamma_{1}\alpha_{3}^{2} \right)  +O(\epsilon^{2}),
\end{align*}
then the matrix $(c_{ij})$ has one pair of simple complex conjugate eigenvalues
\begin{align*}
	\lambda_{c}=\alpha_{3} + {\rm i} \beta_{3}, \qquad \beta_{3}=\sqrt{c_{11}c_{22}-c_{12}c_{21}-\alpha_{3}^{2}},
\end{align*}
near the curve $H^{\star}$.
According to the procedure of analyzing the Hopf bifurcation, we make the change
\begin{align*}
	\eta=T u, \qquad
	T=
	\left(
	\begin{array}{cc}
		2c_{12} & 2c_{12} \\
		-c_{11}+c_{22}+2{\rm i} \beta_{3} & -c_{11}+c_{22}-2{\rm i} \beta_{3} \\
	\end{array}
	\right),
\end{align*}
then system \eqref{3} can be written as
\begin{align}\label{4}
	\left\{
	\begin{array}{ll}
		\dot{u}_{1}=&2\epsilon \left( \lambda_{c}u_{1}+d_{11}u_{1}^{2}+d_{12}u_{1}u_{2}+d_{13}u_{2}^{2}+\epsilon(D_{11}u_{1}^{3}+D_{12}u_{1}^{2}u_{2}+D_{13}u_{1}u_{2}^{2}+D_{14}u_{2}^{3})\right. ,\\
		&\left. +\sum_{4 \leq l+m \leq 6}\epsilon^{l+m-2}\gamma_{lm}^{\star}u_{1}^{l}u_{2}^{m}\right)  +\epsilon^{\frac{11}{2}}f_{31}(u,\varphi,\alpha_{3},\zeta_{2}),\\
		\dot{u}_2=&\bar{\dot{u}}_{1},\\
		\dot{\varphi}_1=&\beta_{1}^{(3)} +\epsilon b_{11}^{(3)}u_{1}+\epsilon b_{12}^{(3)}u_{2}+\sum_{2 \leq l+m \leq 5}\epsilon^{l+m}B_{1lm}^{(3)}u_{1}^{l}u_{2}^{m}+\epsilon^{\frac{11}{2}}f_{33}(u,\varphi,\alpha_{3},\zeta_{2}),\\
		\dot{\varphi}_2=&\beta_{2}^{(3)} +\epsilon b_{21}^{(3)}u_{1}+\epsilon b_{22}^{(3)}u_{2}+\sum_{2 \leq l+m \leq 5}\epsilon^{l+m}B_{2lm}^{(3)}u_{1}^{l}u_{2}^{m}+\epsilon^{\frac{11}{2}}f_{34}(u,\varphi,\alpha_{3},\zeta_{2}),
	\end{array}
	\right.
\end{align}
where $u_{2}=\bar{u}_{1}$,
\begin{align*}
	\beta_{j}^{(3)}=\tau_{0}\omega_{j}+2\epsilon(b_{j1}-\beta_{j1}-\sigma\beta_{j2})\alpha_{3}+\epsilon(b_{j1}+b_{j2}+(\gamma-1)\beta_{j1}+(1-\sigma)\beta_{j2})\zeta_{2}+\epsilon^{2}X(\alpha_{3},\zeta_{2}), \qquad j=1,2,
\end{align*}
and the coefficients in \eqref{4} are seen in Appendix A.4.
It follows from the normal form method that there is an invertible parameter-dependent change of complex coordinate
\begin{align*}
	u_{1} \rightarrow u_{1}+\sum_{2 \leq l+m \leq 8}h_{lm}u_{1}^{l} u_{2}^{m},
\end{align*}
with $u_{2}=\bar{u}_{1}$ and $h_{21}=h_{32}=h_{43}=0$, such that we can normalize \eqref{4} up to the eighth-order term, that is
\begin{align}\label{6}
	\left\{
	\begin{array}{ll}
		\dot{u}_{1}=2\epsilon\left( \lambda_{c}u_{1}+l_{21}u_{1}^{2}u_{2}+l_{32}u_{1}^{3}u_{2}^{2}+l_{43}u_{1}^{4}u_{3}^{3}+O(\left\| u \right\|^{9})\right) +\epsilon^{\frac{11}{2}}f_{41}(u,\varphi,\alpha_{3},\zeta_{2}).\\
		\dot{u}_2=&\bar{\dot{u}}_{1},\\
		\dot{\varphi}_{1}=\beta_{1}^{(3)} +\epsilon g_{11}(u,\alpha_{3},\zeta_{2})+\epsilon^{\frac{11}{2}}f_{43}(u,\varphi,\alpha_{3},\zeta_{2}),\\
		\dot{\varphi}_{2}=\beta_{2}^{(3)} +\epsilon g_{12}(u,\alpha_{3},\zeta_{2})+\epsilon^{\frac{11}{2}}f_{44}(u,\varphi,\alpha_{3},\zeta_{2}),
	\end{array}
	\right.
\end{align}
where $g_{1j}=O(\left\| u \right\|)$.
Moreover, we list some expressions of $l_{(i+1)i}(\alpha_{3})$ in Appendix A.5.
\par
Using polar coordinates $(r_{3},\varphi_{3})$ with $u_{1}=r_{3}e^{{\rm i} \varphi_{3}}$, $u_{2}=r_{3}e^{-{\rm i} \varphi_{3}}$, letting $\rho_{3}=r_{3}^{2}$, and rescaling
\begin{align*}
	\rho_{3} \rightarrow \epsilon\rho_{3}, \qquad \alpha_{3} \rightarrow \epsilon^{2}\alpha_{3},
\end{align*}
system \eqref{6} becomes
\begin{align}\label{8}
	\left\{
	\begin{array}{ll}
		\dot{\rho}_{3}=4\epsilon^{3}\rho_{3} \left( \alpha_{3} +a_{1}^{'}\rho_{3} +\epsilon a_{2}^{'}\rho_{3}^{2} +\epsilon^{2} a_{3}^{'}\rho_{3}^{3} \right) +\epsilon^{5}f_{51}(\rho_{3},\varphi,\alpha_{3},\zeta_{2}),\\
		\dot{\varphi}_{3}=2 \epsilon \left( \beta_{3}^{(4)} +\epsilon b_{1}^{'}\rho_{3} +\epsilon^{2} b_{2}^{'}\rho_{3}^{2} +\epsilon^{3}b_{3}^{'}\rho_{3}^{3}\right) +\epsilon^{5}f_{52}(\rho_{3},\varphi,\alpha_{3},\zeta_{2}),\\
		\dot{\varphi}_{1}=\beta_{1}^{(4)} +\epsilon^{\frac{3}{2}} g_{21}(\rho_{3},\varphi_{3},\alpha_{3},\zeta_{2})+\epsilon^{\frac{11}{2}}f_{53}(\rho_{3},\varphi,\alpha_{3},\zeta_{2}),\\
		\dot{\varphi}_{2}=\beta_{2}^{(4)} +\epsilon^{\frac{3}{2}} g_{22}(\rho_{3},\varphi_{3},\alpha_{3},\zeta_{2})+\epsilon^{\frac{11}{2}}f_{54}(\rho_{3},\varphi,\alpha_{3},\zeta_{2}),
	\end{array}
	\right.
\end{align}
where $\varphi=(\varphi_{1},\varphi_{2},\varphi_{3})^{T}$, $g_{2j}=O(\left\| \rho_{3} \right\|^{\frac{1}{2}})$, $a_{1}^{'}=a_{10}^{'}\zeta_{2}^{2} +O(\epsilon)$, $b_{1}^{'}=b_{10}^{'}\zeta_{2} +O(\epsilon)$,
\begin{align*}
	\beta_{j}^{(4)}=\tau_{0}\omega_{j}+2\epsilon^{3}(b_{j1}-\beta_{j1}-\sigma\beta_{j2})\alpha_{3}+\epsilon(b_{j1}+b_{j2}+(\gamma-1)\beta_{j1}+(1-\sigma)\beta_{j2})\zeta_{2}+\epsilon^{2}X(\epsilon^{2}\alpha_{3},\zeta_{2}),
\end{align*}
\begin{align*}
	(\beta_{3}^{(4)})^{2}=3(\zeta_{2}^{2}+2\epsilon^{2}\alpha_{3}\zeta_{2})-\epsilon^{4}\alpha_{3}^{2} +\epsilon X(\epsilon^{2}\alpha_{3},\zeta_{2}), \qquad j=1,2,
\end{align*}
with
\begin{align*}
	a_{10}^{'}=&\frac{2 \gamma }{3}(-4(\Gamma_{1}+\Gamma_{2}+\Sigma_{2}+\Sigma_{3}) +(7\Gamma_{1}+\Sigma_{1}+\Sigma_{2}-\gamma\Sigma_{1})\gamma +(\Gamma_{2}+\Gamma_{3}+7\Sigma_{3}-\sigma\Gamma_{3})\sigma)(0),\\
	b_{10}^{'}=&\frac{2 \gamma ^2 \sigma  (-\gamma -\sigma +2)}{3 \sqrt{3}} (0).
\end{align*}
The first equation of \eqref{8} without the team $\epsilon^{5}f_{51}$ has a positive equilibrium $\rho_{30}=-\frac{\alpha_{3}}{a_{1}^{'}}+O(\epsilon)$ for sufficiently small $\epsilon$ and $\alpha_{3} a_{1}^{'}<0$.
System \eqref{8} by the change
\begin{align*}
	\rho_{3}=\rho_{30}+\epsilon^{\frac{1}{2}}I_{3},
\end{align*}
is transformed into the form
\begin{align}\label{10}
	\left\{
	\begin{array}{ll}
		\dot{I}_{3}=4\epsilon^{3}\left( -\alpha_{3}I_{3}+\epsilon^{\frac{1}{2}}a_{1}^{'}I_{3}^{2}+\epsilon f_{61}(I_{3},\varphi,\alpha_{3},\zeta_{2}) \right),\\
		\dot{\varphi}_{3}=2 \epsilon \left( \beta_{3}^{(5)} +\epsilon^{\frac{3}{2}}(b_{1}^{'}+2\epsilon\rho_{30}b_{2}^{'})I_{3}+\epsilon^{3}h_{30}(I_{3},\varphi_{3},\alpha_{3},\zeta_{2}) +\epsilon^{4}f_{62}(I_{3},\varphi,\alpha_{3},\zeta_{2}) \right),\\
		\dot{\varphi}_{1}=\beta_{1}^{(5)} +\epsilon^{\frac{3}{2}} g_{31}(\varphi_{3},\alpha_{3},\zeta_{2})+\epsilon^{2} h_{31}(I_{3},\varphi_{3},\alpha_{3},\zeta_{2})+\epsilon^{\frac{11}{2}}f_{63}(I_{3},\varphi,\alpha_{3},\zeta_{2}),\\
		\dot{\varphi}_{2}=\beta_{2}^{(5)} +\epsilon^{\frac{3}{2}} g_{32}(\varphi_{3},\alpha_{3},\zeta_{2})+\epsilon^{2} h_{32}(I_{3},\varphi_{3},\alpha_{3},\zeta_{2})+\epsilon^{\frac{11}{2}}f_{64}(I_{3},\varphi,\alpha_{3},\zeta_{2}),
	\end{array}
	\right.
\end{align}
where $g_{3j}(\varphi_{3},\alpha_{3},\zeta_{2})$ are analytic in $\varphi_{3}$ with the mean value zero,
\begin{align*}
	\beta_{1}^{(5)}&=\tau_{0}\omega_{1}+\epsilon\omega_{11}\zeta_{2}+\epsilon^{3}\omega_{12}\alpha_{3}+\epsilon^{2}X(\epsilon^{2}\alpha_{3},\zeta_{2}),\\
	\beta_{2}^{(5)}&=\tau_{0}\omega_{2}+\epsilon\omega_{21}\zeta_{2}+\epsilon^{3}\omega_{22}\alpha_{3}+\epsilon^{2}X(\epsilon^{2}\alpha_{3},\zeta_{2}),\\
	\beta_{3}^{(5)}&=\sqrt{3}\zeta_{2}+\epsilon^{2}\sqrt{3}\alpha_{3}-\epsilon \frac{b_{10}^{'}}{a_{10}^{'}}\frac{\alpha_{3}}{\zeta_{2}}+\epsilon X(\epsilon^{2}\alpha_{3},\zeta_{2})+\epsilon^{2} X(\alpha_{3},\zeta_{2})
\end{align*}
with
\begin{align*}
	\omega_{11}&=\left(  -\frac{p_{22}}{p_{11}}(q_{11}-\beta_{11}p_{11}-\beta_{12}p_{21})+(q_{12}-\beta_{11}p_{12}-\beta_{12}p_{22}) \right) (0),\\
	\omega_{21}&=\left(  	-\frac{p_{22}}{p_{11}}(q_{21}-\beta_{21}p_{11}-\beta_{22}p_{21})+(q_{22}-\beta_{21}p_{12}-\beta_{22}p_{22}) \right) (0),
\end{align*}
\begin{align*}
	\omega_{12}=&2(b_{11}-\beta_{11}-\sigma\beta_{12})-\frac{8}{a_{10}^{'}}\left( \frac{1}{3}(\gamma b_{11}+b_{12})\left(-\gamma ^2 \Sigma _1+\gamma  \Gamma _1-4 \Sigma _3+\gamma  \Sigma _2+\Gamma _3 \sigma -\Gamma _2\right)\right. \\
	&\left. -\left(-\gamma  \Gamma _1-\Gamma _3 \sigma +\Gamma _2\right)b_{12}+\gamma\left( \gamma B_{120}+B_{111}+\sigma B_{102}\right)  \right)(0),\\
	\omega_{22}=&2(b_{21}-\beta_{21}-\sigma\beta_{22})-\frac{8}{a_{10}^{'}}\left( \frac{1}{3}(\gamma b_{21}+b_{22})\left(-\gamma ^2 \Sigma _1+\gamma  \Gamma _1-4 \Sigma _3+\gamma  \Sigma _2+\Gamma _3 \sigma -\Gamma _2\right)\right. \\
	&\left. -\left(-\gamma  \Gamma _1-\Gamma _3 \sigma +\Gamma _2\right)b_{22}+\gamma\left( \gamma B_{220}+B_{211}+\sigma B_{202}\right)  \right)(0).
\end{align*}
\par
Let $\Pi_{v}$ denote the parameter region where $\rho_{30}>0$, and take a convex bounded open subset $\Pi_{v}^{0} \subset \Pi_{v}$ with positive Lebesgue measure.
Assuming
\begin{align}\label{C0}
	\left( \omega_{11}\omega_{22}-\omega_{12}\omega_{21}\right) \left( \omega_{1}\omega_{21}-\omega_{2}\omega_{11} \right) \neq 0
\end{align}
holds, the following results can be obtained by applying KAM technique to \eqref{10}, noting the following facts:

1)By proof of theorem 2 in \cite{L13}, system \eqref{10} can be changed to this case where $n_{12}=0$, $n_{11}=n_{21}=1$, $n_{22}=2$, $q_{1}=3$, $q_{2}=\frac{1}{2}$, $q_{5}=1$, $q_{6}=\frac{3}{2}$, $q_{7}=2$ in (1.1) in \cite{L12};

2)System \eqref{10} is $C^{r}$, ($r$ is any positive integer);

3) The analytic part can be proven by Lemma 4.1 in \cite{L12}, and the proof of the geometric part is similar to that of Theorem 2 in \cite{L13}.

\begin{theorem}\label{theorem4}
	If \eqref{C0} holds, then system \eqref{10} has quasi-periodic invariant 3-tori for most parameter values of $\Pi_{v}^{0}$ and sufficiently small $\epsilon$.
\end{theorem}
\begin{corollary}\label{corollary2}
	For any fixed $(\varepsilon_{0},\tau_{0})$ with the corresponding eigenvalues $\pm {\rm i} \omega_{1}$, $\pm {\rm i} \omega_{2}$ satisfying \eqref{ceq}, in addition to the assumptions \eqref{H0}, \eqref{H1}, \eqref{H3}, \eqref{HH0}, \eqref{HH78}, \eqref{HH9}, if \eqref{C1} and \eqref{C0} hold, then system \eqref{fulleq} has quasi-periodic invariant 3-tori for most parameter values near Hopf bifurcation curve H and in the direction of the Hopf bifurcation.
\end{corollary}

\section{Numerical simulations}\label{sec4}
In this section, we shall present some numerical examples of systems \eqref{fulleq} and \eqref{cen-mani-eq} with $c=d=0$ to verify our main theoretical analysis.
In fact, the invariant 2-tori and 3-tori of the system \eqref{cen-mani-eq} are the same as the invariant 2-tori and 3-tori of the system \eqref{fulleq}.
However, the variables of system \eqref{cen-mani-eq} are complex numbers, so we first change them into real numbers.
Letting $y_{i}={\rm Re}\left\lbrace z_{i} \right\rbrace$, $y_{i+1}={\rm Im}\left\lbrace z_{i} \right\rbrace$, $i=1,3$, system \eqref{cen-mani-eq} is changed into
\begin{align}\label{last}
	\dot{y}=J^{*} y+T_{0} g(T_{0}^{-1}y,\alpha);
\end{align}
where
\begin{align*}
	J^{*}=\left(
	\begin{array}{cccc}
		0 & {\rm i} \tau_{0} \omega_{1} & 0 & 0 \\
		{\rm i} \tau_{0} \omega_{1} & 0 & 0 & 0 \\
		0 & 0 & 0 & {\rm i} \tau_{0} \omega_{2} \\
		0 & 0 & {\rm i} \tau_{0} \omega_{2} & 0 \\
	\end{array}
	\right), \qquad
	T_{0}=\left(
	\begin{array}{cccc}
		\frac{1}{2} & \frac{1}{2} & 0 & 0 \\
		\frac{1}{2} & -\frac{1}{2} & 0 & 0 \\
		0 & 0 & \frac{1}{2} & \frac{1}{2} \\
		0 & 0 & \frac{1}{2} & -\frac{1}{2} \\
	\end{array}
	\right).
\end{align*}

\begin{example}
	Let $a=1.2$, $b=-6$.
	When $\varepsilon_{0} \approx 0.2231$ and $\tau_{0} \approx 7.8628$, we already have $\omega_{1} \approx 1.0608$, $\omega_{2} \approx 0.9083$, $p_{11}(0)\approx-7.2397$, $p_{22}(0)\approx-13.6504$, $\gamma(0) \approx 1.0607$, $\sigma(0) \approx 3.7710$, $\alpha_{1} \approx 0.1016 \delta_{1}+0.1252 \delta_{2}$, $\alpha_{2} \approx 0.3708 \delta_{1}-0.2511 \delta_{2}$.
	It is a simple case.
	Moreover, we obtain
	\begin{align*}
		\det \left( \frac{\partial \omega_{1}^{0}(\delta)}{\partial \delta}  \right) \approx -4.2542.
	\end{align*}
	Thus the conditions in Corollary \ref{corollary1} hold for $-1.0041\alpha_{1}<\alpha_{2}<0.6108\alpha_{1}$.
	Figure \ref{figure-2} shows a quasi-periodic invariant 2-torus of system \eqref{fulleq} for $(\alpha_{1},\alpha_{2})=(0.0012,-0.001)$ in region 4 of Figure \ref{2-fzt}, and Figure \ref{figure-2D} shows this quasi-periodic invariant 2-torus disappears when $(\alpha_{1},\alpha_{2})=(-0.009,-0.005)$.
	Figure \ref{Figure2-cm} shows some projections of the quasi-periodic invariant 2-torus of system \eqref{last} when $(\alpha_{1},\alpha_{2})=(0.0012,-0.001)$.
\end{example}
\begin{figure}
	\centering
	\subfloat[]{\includegraphics[width=80mm, height=80mm]{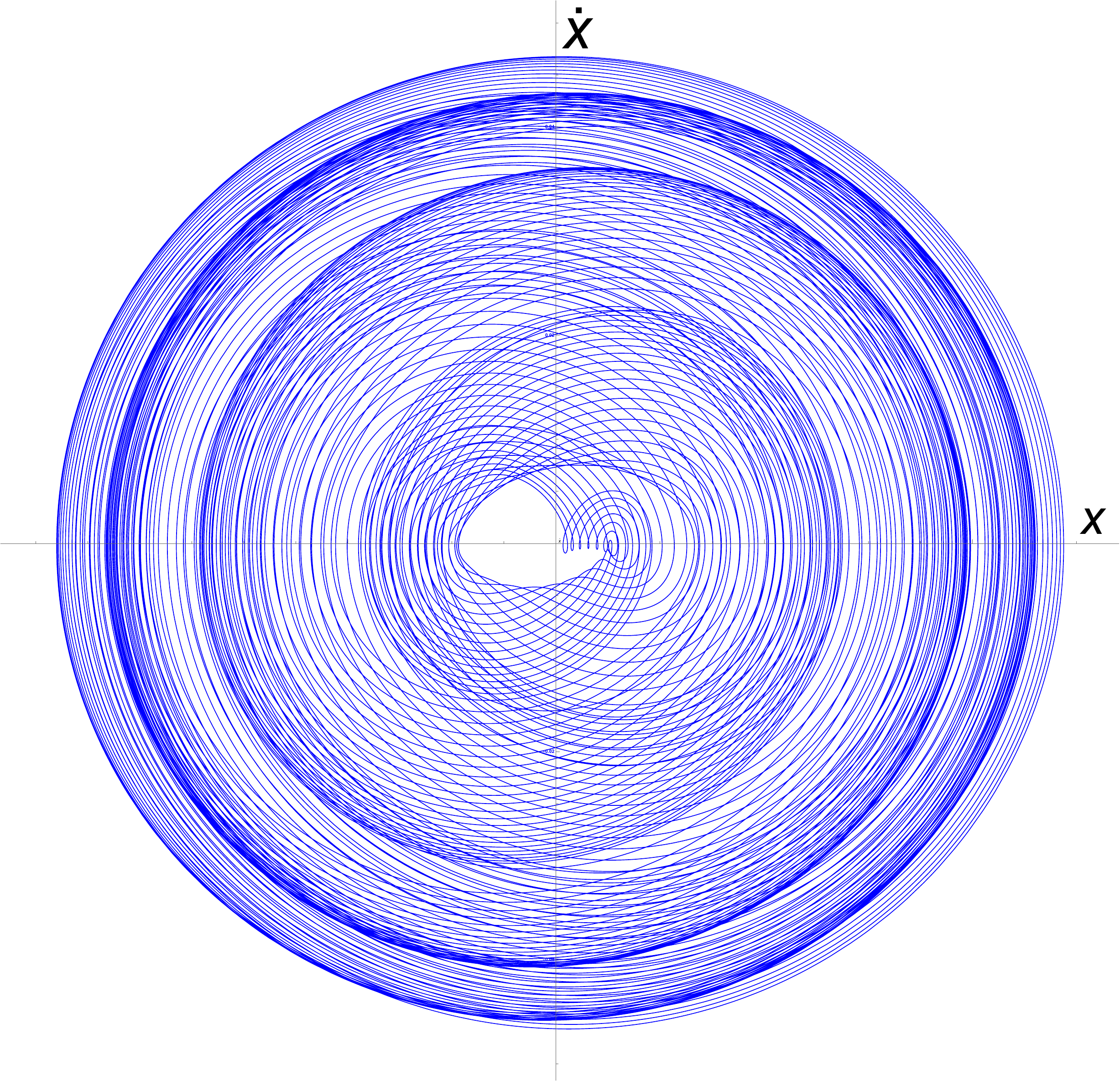}\label{2X}}
	\hfill
	\subfloat[]{\includegraphics[width=150mm, height=65mm]{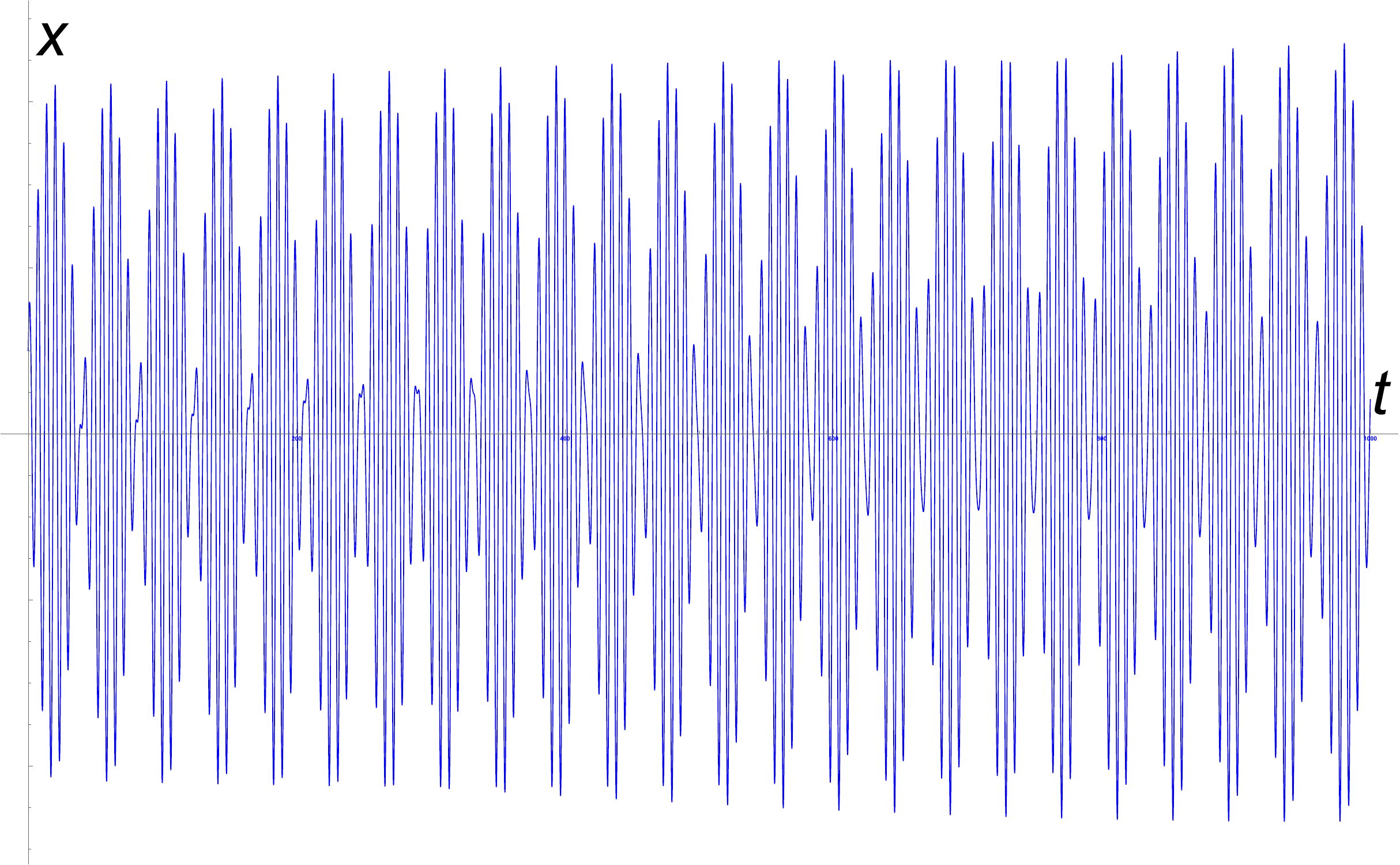}\label{2Y}}
	\caption{When $(\alpha_{1},\alpha_{2})=(0.0012,-0.001)$, $t\in \left[ 0,1000 \right]$, system \eqref{fulleq} with $(x(0),\dot{x}(0))=(0.01,0.01)$ has a quasi-periodic solution: the phase portrait (a) and the time evolution (b).}
	\label{figure-2}
\end{figure}
\begin{figure}
	\centering
	\includegraphics[width=150mm, height=70mm]{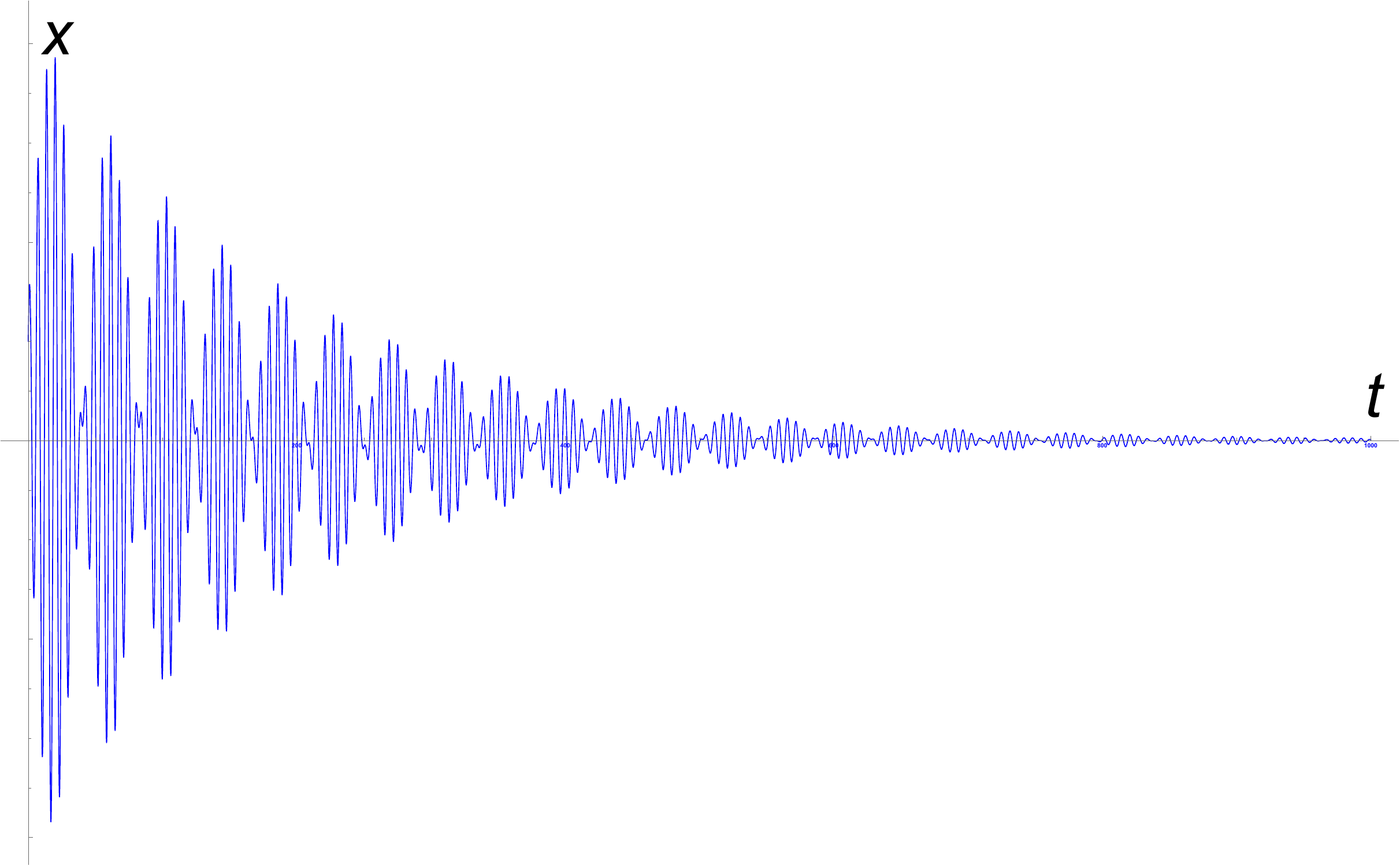}\label{2Z}
	\caption{When $(\alpha_{1},\alpha_{2})=(-0.009,-0.005)$, $t\in \left[ 0,1000 \right]$, the evolution of solution of system \eqref{fulleq} with $(x(0),\dot{x}(0))=(0.01,0.01)$.}
	\label{figure-2D}
\end{figure}
\begin{figure}[htbp]
	\centering
	\begin{minipage}{0.49\linewidth}
		\centering
		\includegraphics[width=80mm]{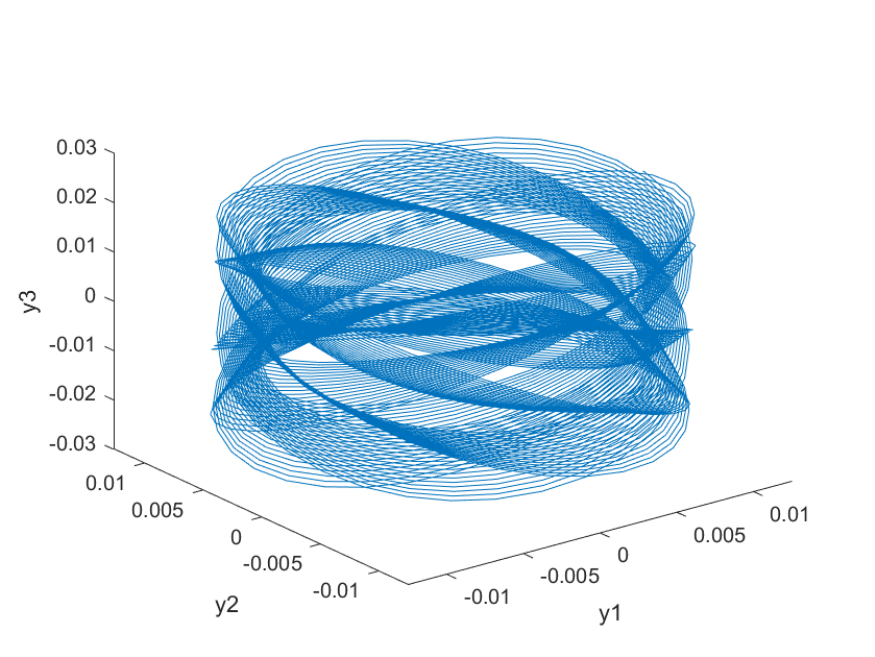}
	\end{minipage}
	\begin{minipage}{0.49\linewidth}
		\centering
		\includegraphics[width=80mm]{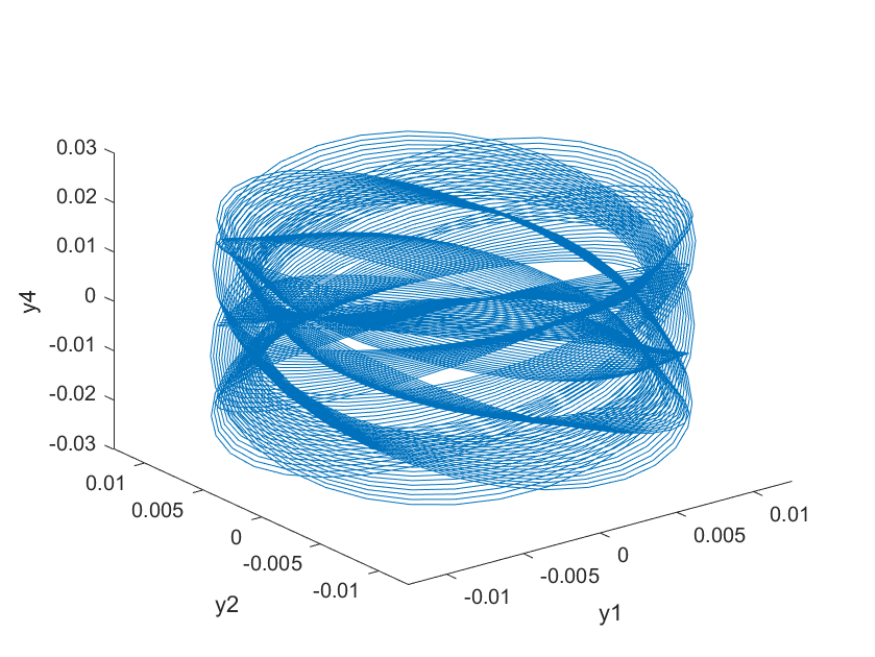}
	\end{minipage}

	\begin{minipage}{0.49\linewidth}
		\centering
		\includegraphics[width=80mm]{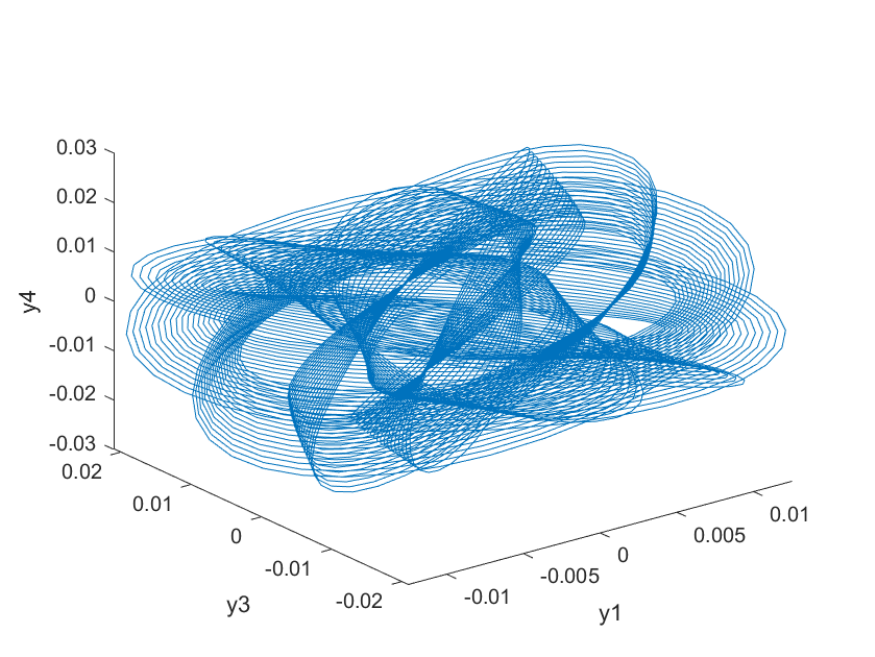}
	\end{minipage}
	\begin{minipage}{0.49\linewidth}
		\centering
		\includegraphics[width=80mm]{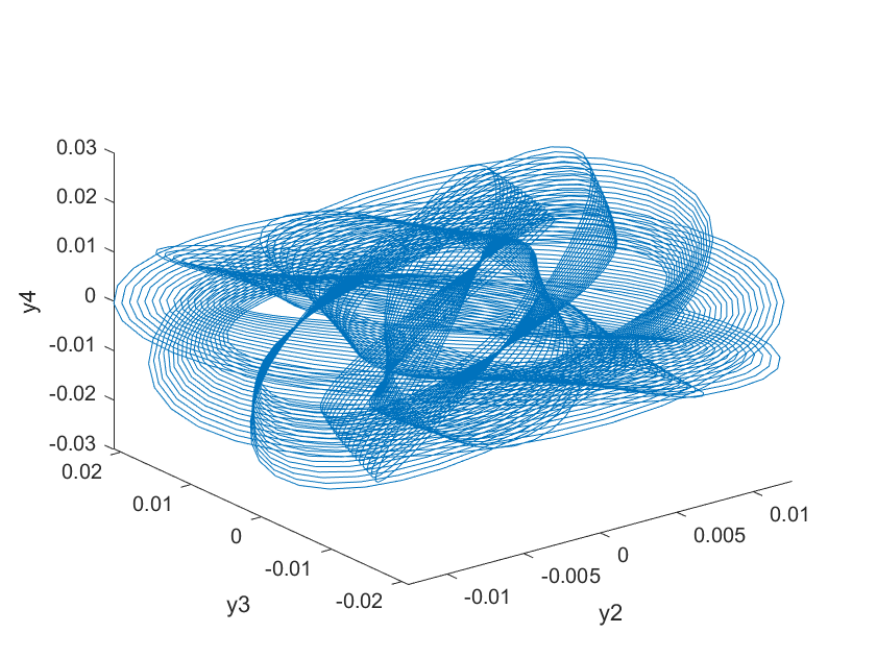}
	\end{minipage}

	\begin{minipage}{0.3\linewidth}
		\centering
		\includegraphics[width=55mm]{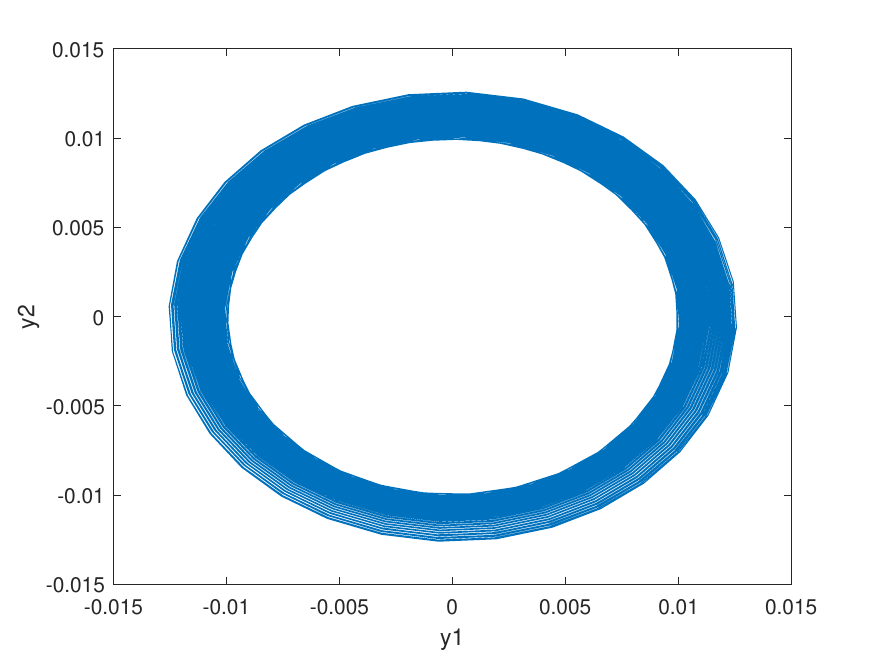}
	\end{minipage}
	\begin{minipage}{0.3\linewidth}
		\centering
		\includegraphics[width=55mm]{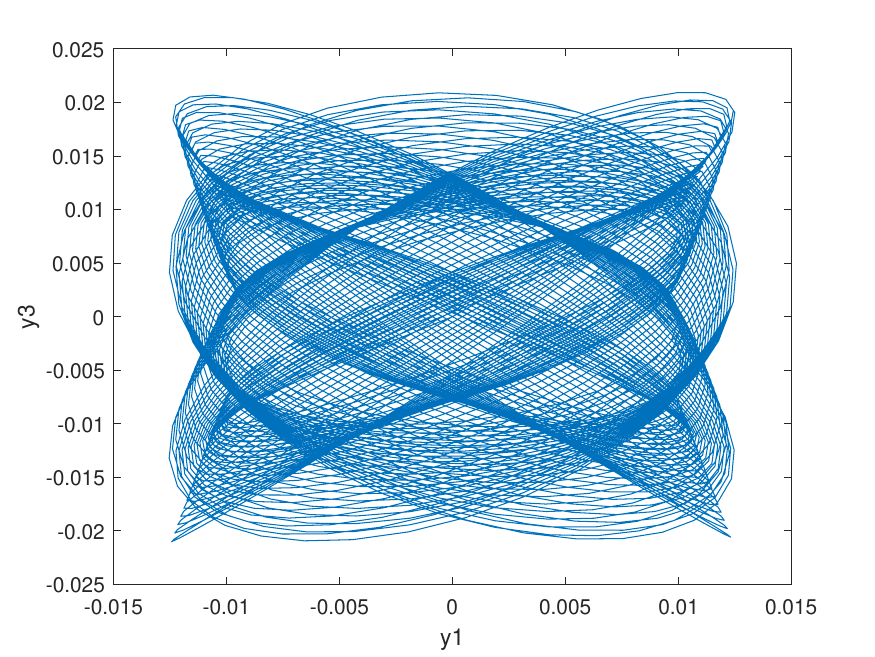}
	\end{minipage}
	\begin{minipage}{0.3\linewidth}
		\centering
		\includegraphics[width=55mm]{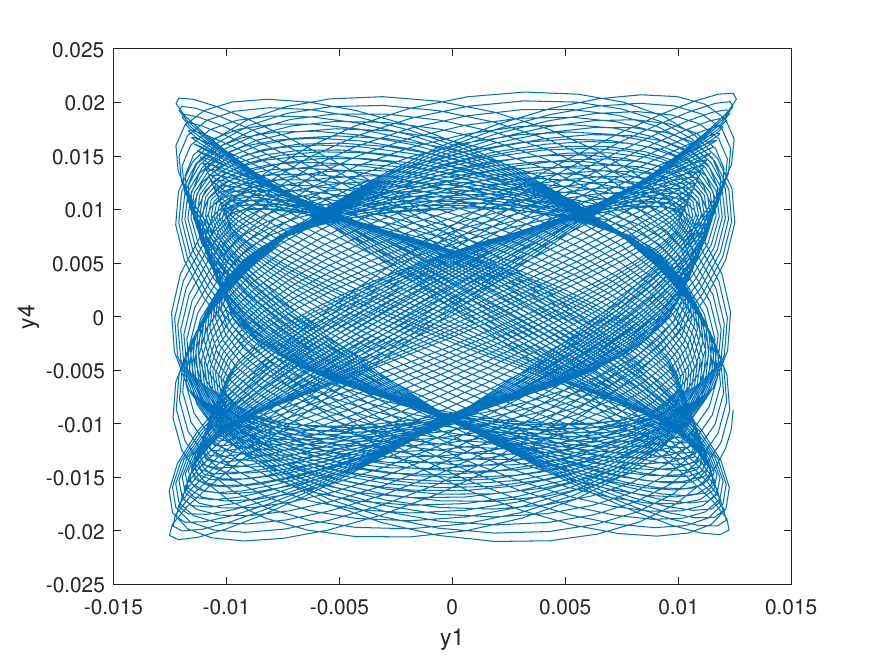}
	\end{minipage}

	\begin{minipage}{0.3\linewidth}
		\centering
		\includegraphics[width=55mm]{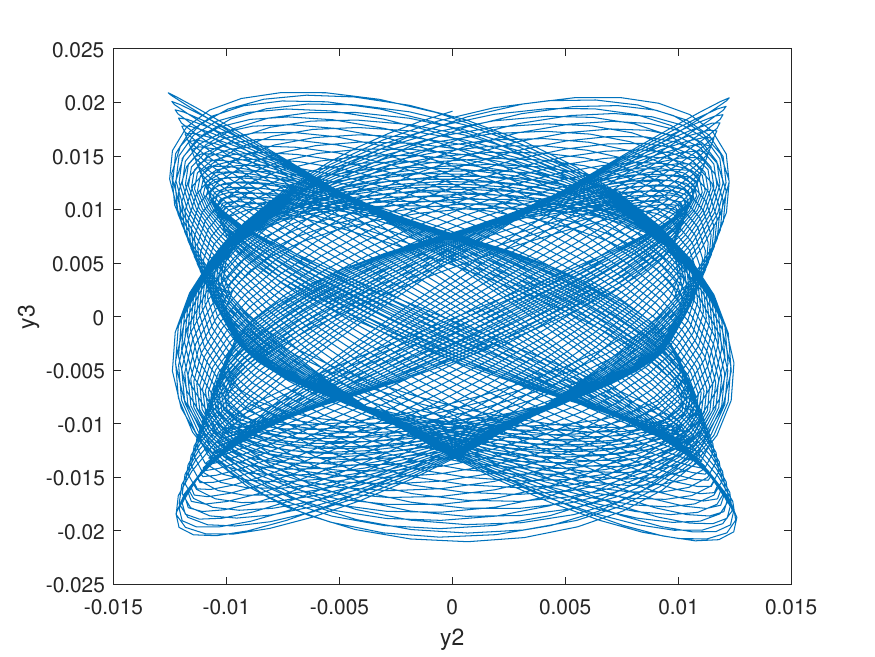}
	\end{minipage}
	\begin{minipage}{0.3\linewidth}
		\centering
		\includegraphics[width=55mm]{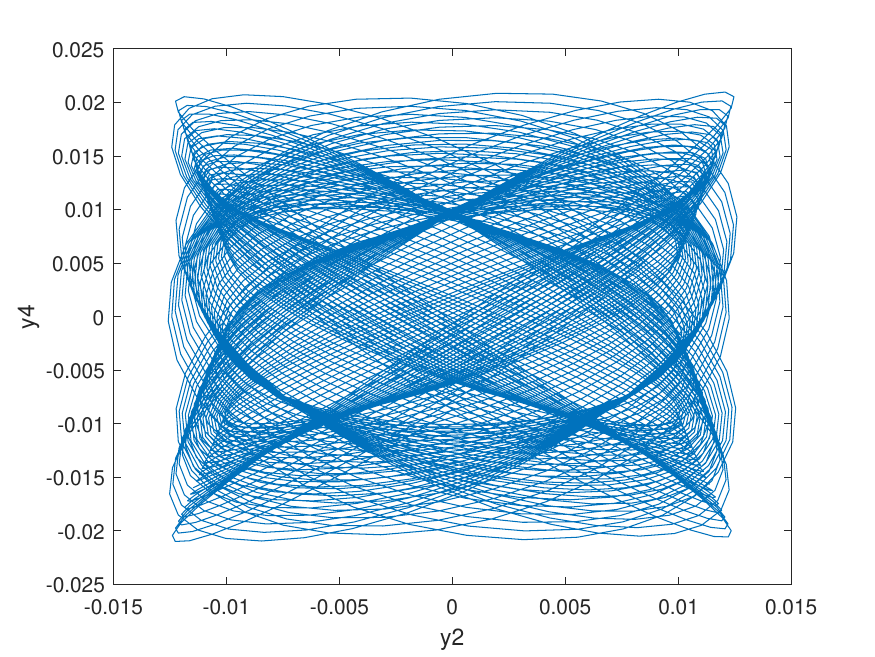}
	\end{minipage}
	\begin{minipage}{0.3\linewidth}
		\centering
		\includegraphics[width=55mm]{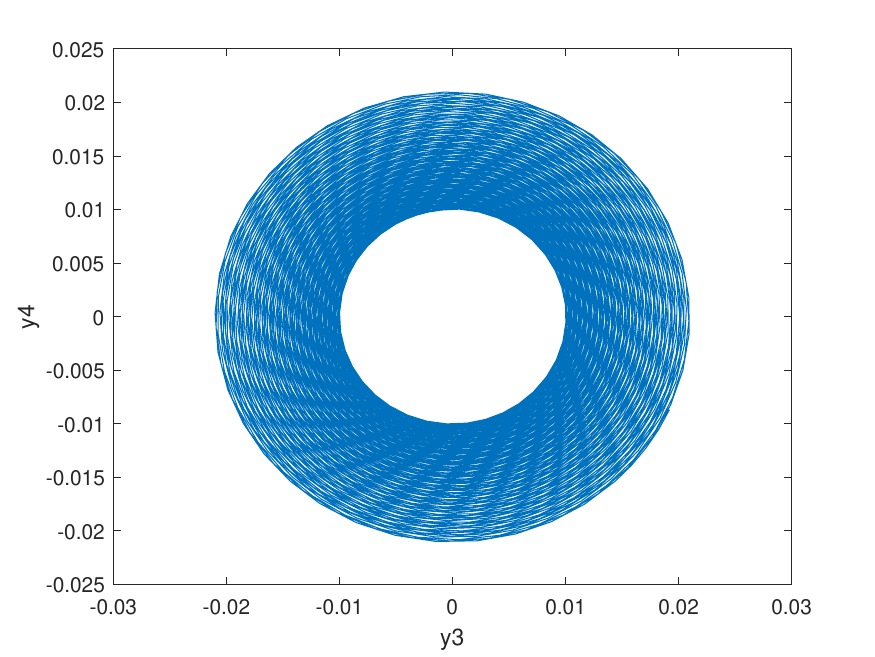}
	\end{minipage}
	\caption{When $(\alpha_{1},\alpha_{2})=(0.0012,-0.001)$ and $t\in \left[ 0,100 \right]$, the phase portraits of system \eqref{last}.}
	\label{Figure2-cm}
\end{figure}

\begin{example}
	Let $a=1$, $b=\frac{1}{7}$.
	When $\varepsilon_{0} \approx 0.2533$ and $\tau_{0}=2.5\pi$, we already have $\omega_{1}=1$, $\omega_{2} \approx 0.9674$, $p_{11}(0)\approx -0.4837$, $p_{22}(0)\approx 0.3976$, $\gamma(0) \approx -2.4333$, $\sigma(0) \approx -1.6439$, $\alpha_{1} \approx 0.1207 \delta_{1}+0.1262 \delta_{2}$, $\alpha_{2} \approx 0.1276 \delta_{1}$, $l_{10}\approx-6591.28$.
	It is a difficult case.
	Furthermore, we obtain $\omega_{1}\omega_{21}-\omega_{2}\omega_{11} \approx -59.9828$ and $\omega_{11}\omega_{22}-\omega_{12}\omega_{21} \approx -1097.82$.
	According to Corollary \ref{corollary2}, it follows that system \eqref{fulleq} has a stable quasi-periodic invariant 3-tori for most parameter values $(\alpha_{1},\alpha_{2})$ satisfying $0.1846\alpha_{2}-876.087\alpha_{2}^{2}<\alpha_{1}<0.1846\alpha_{2}$.
	Figure \ref{figure-3} shows a quasi-periodic invariant 3-torus of system \eqref{fulleq} for $(\alpha_{1},\alpha_{2})=(-0.00006,-0.00024)$ in region 6 of Figure \ref{3-fzt}.
	Since the quasi-periodic invariant torus can not be shown well on a plane, Figure \ref{Figure3-cm} shows some projections of a quasi-periodic invariant 3-torus of system \eqref{last} when $(\alpha_{1},\alpha_{2})=(-0.00006,-0.00024)$.
\end{example}
\begin{figure}
	\centering
	\subfloat[]{\includegraphics[width=80mm, height=80mm]{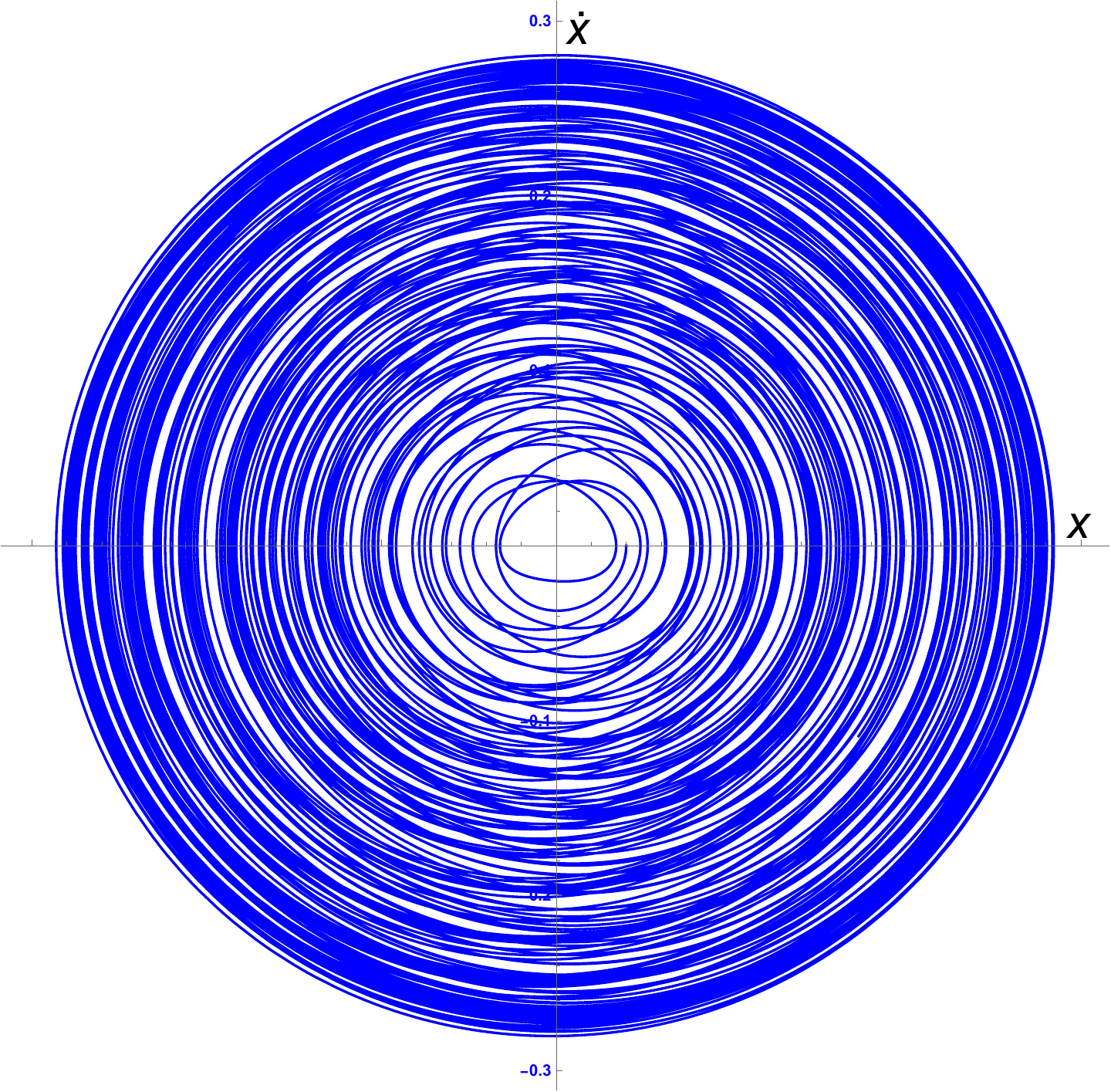}\label{2X}}
	\hfill
	\subfloat[]{\includegraphics[width=150mm, height=65mm]{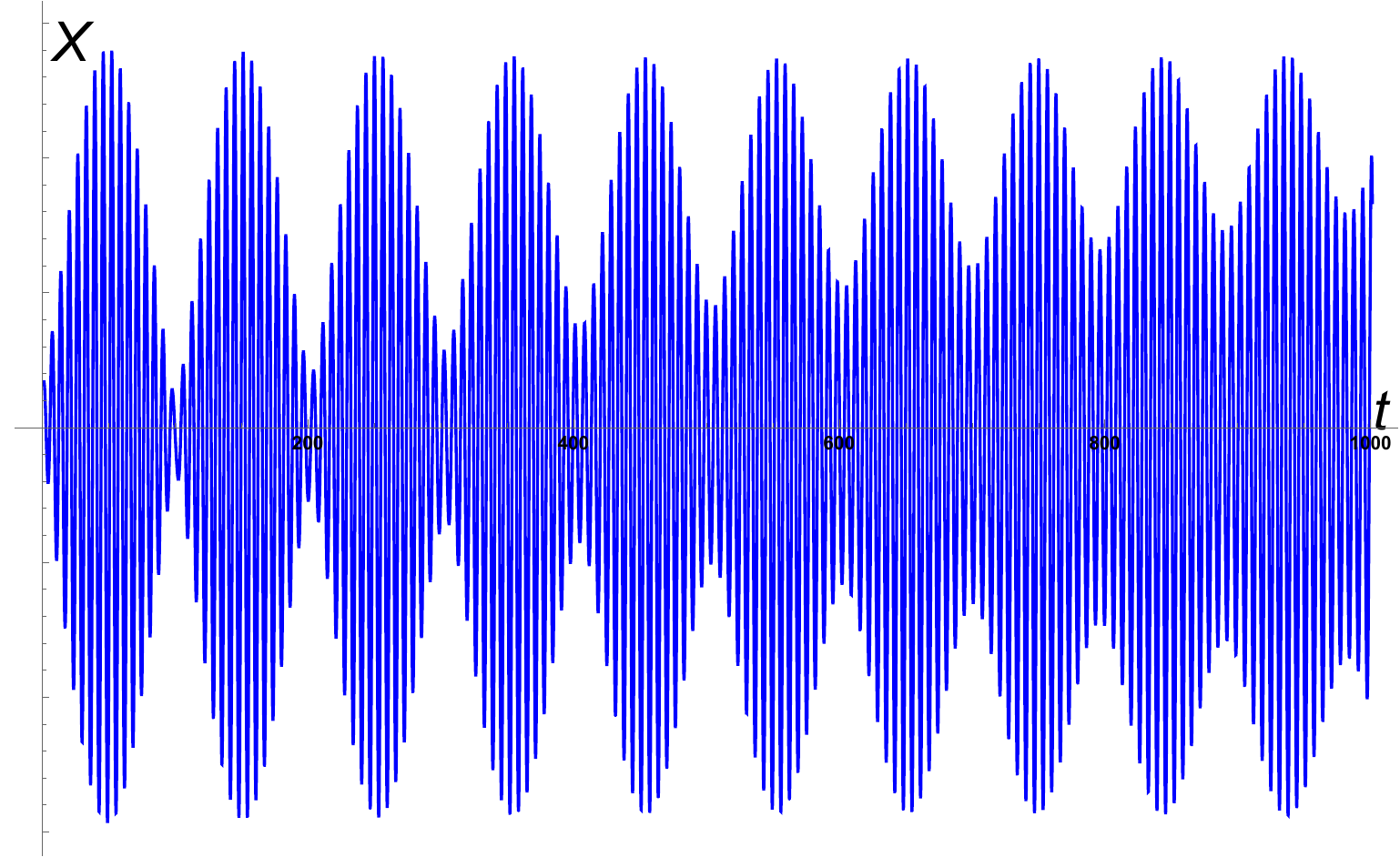}\label{2Y}}
	\caption{When $(\alpha_{1},\alpha_{2})=(-0.00006,-0.00024)$, $t\in \left[ 0,1000 \right]$, system \eqref{fulleq} with $(x(0),\dot{x}(0))=(0.04,0)$ has a quasi-periodic solution: the phase portrait (a) and the time evolution (b).}
	\label{figure-3}
\end{figure}
\begin{figure}[htbp]
	\centering
	\begin{minipage}{0.49\linewidth}
		\centering
		\includegraphics[width=70mm]{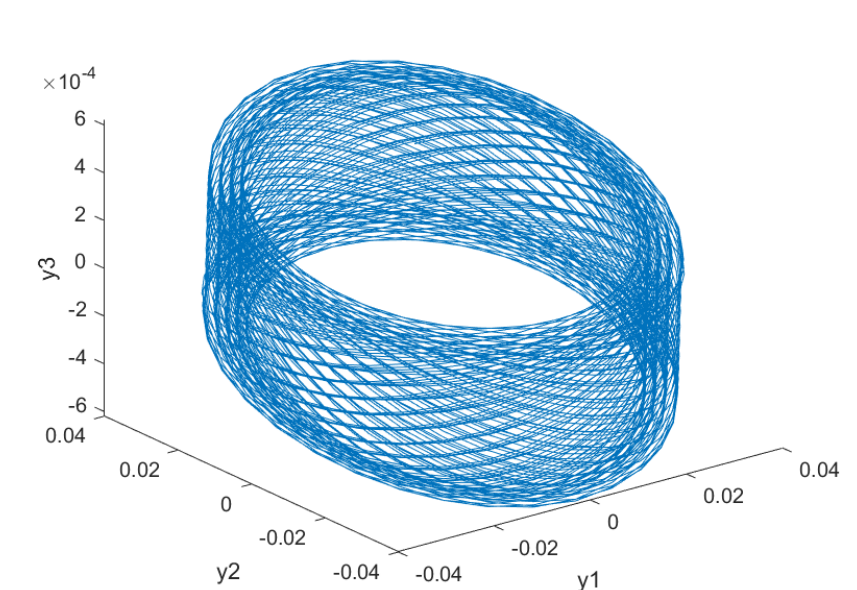}
	\end{minipage}
	\begin{minipage}{0.49\linewidth}
		\centering
		\includegraphics[width=70mm]{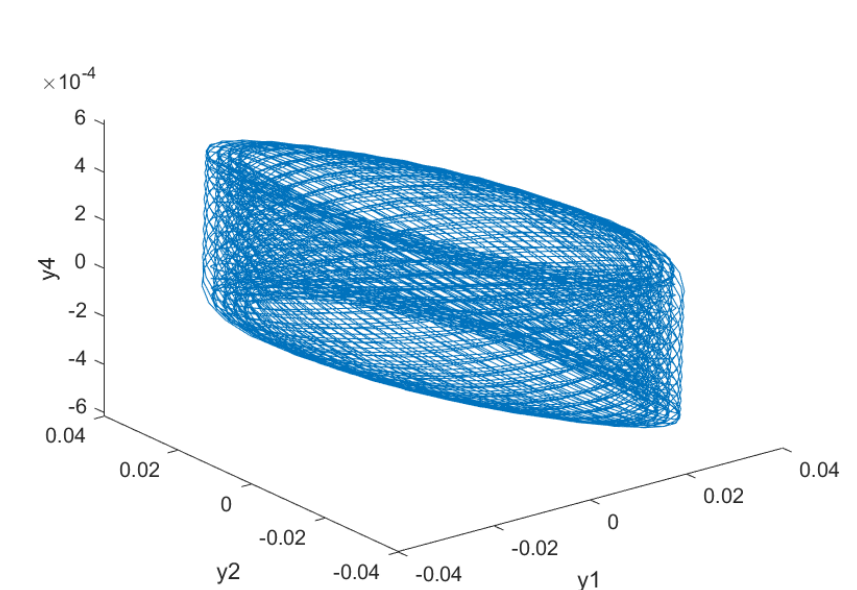}
	\end{minipage}
	
	\begin{minipage}{0.49\linewidth}
		\centering
		\includegraphics[width=70mm]{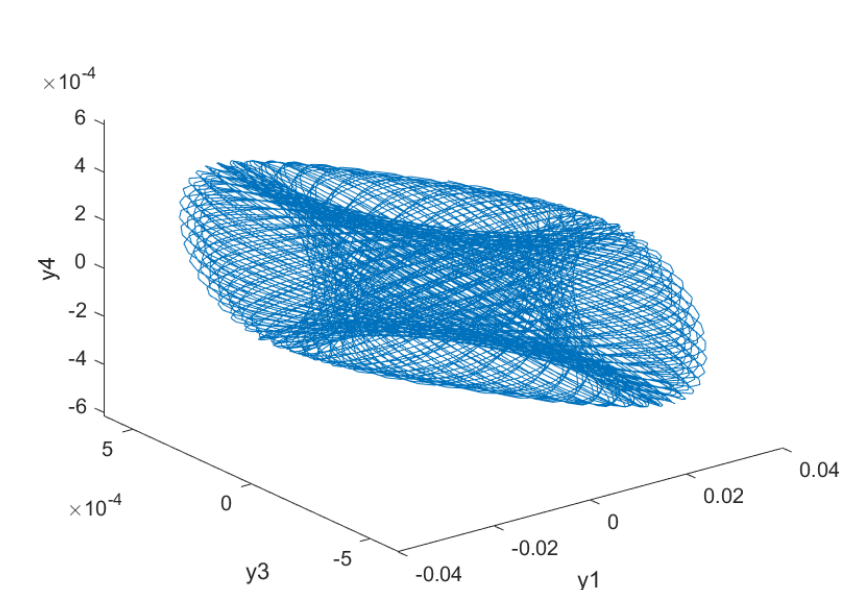}
	\end{minipage}
	\begin{minipage}{0.49\linewidth}
		\centering
		\includegraphics[width=70mm]{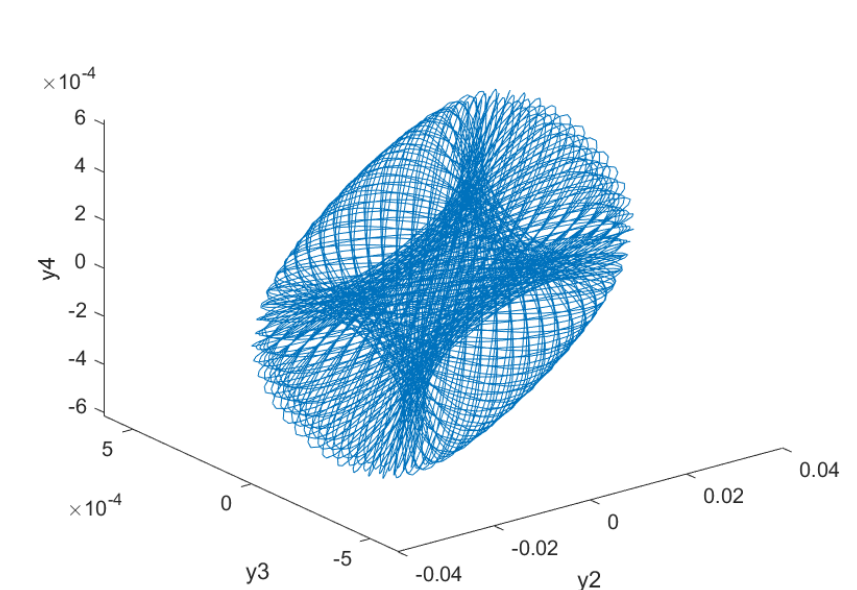}
	\end{minipage}
	
	\begin{minipage}{0.3\linewidth}
		\centering
		\includegraphics[width=55mm]{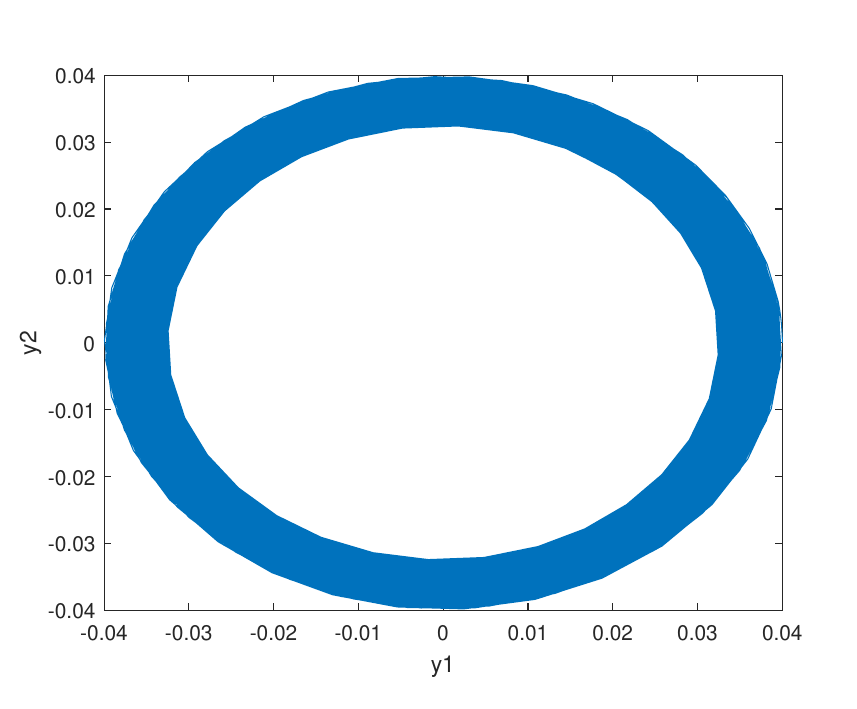}
	\end{minipage}
	\begin{minipage}{0.3\linewidth}
		\centering
		\includegraphics[width=55mm]{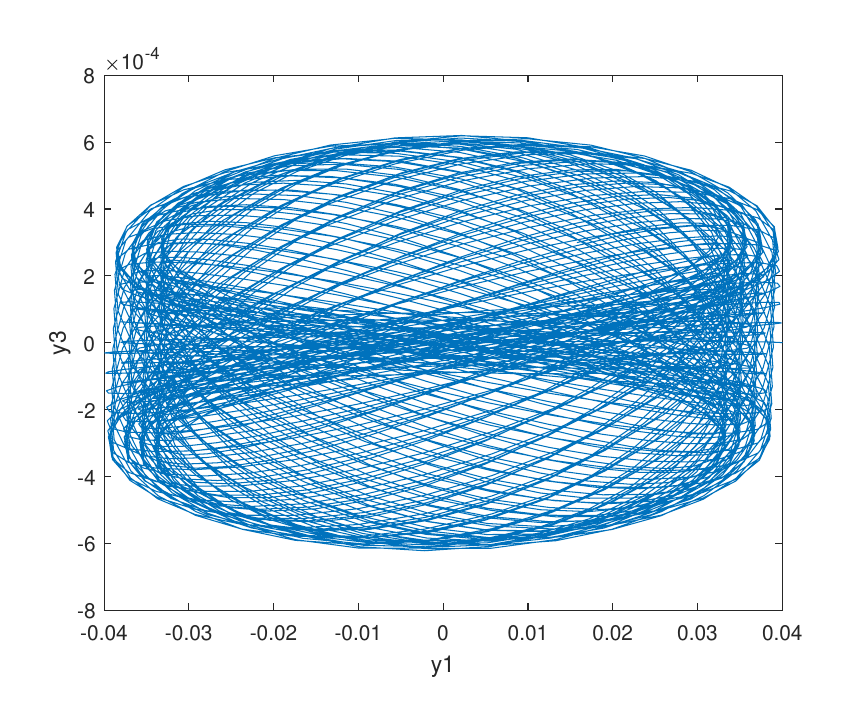}
	\end{minipage}
	\begin{minipage}{0.3\linewidth}
		\centering
		\includegraphics[width=55mm]{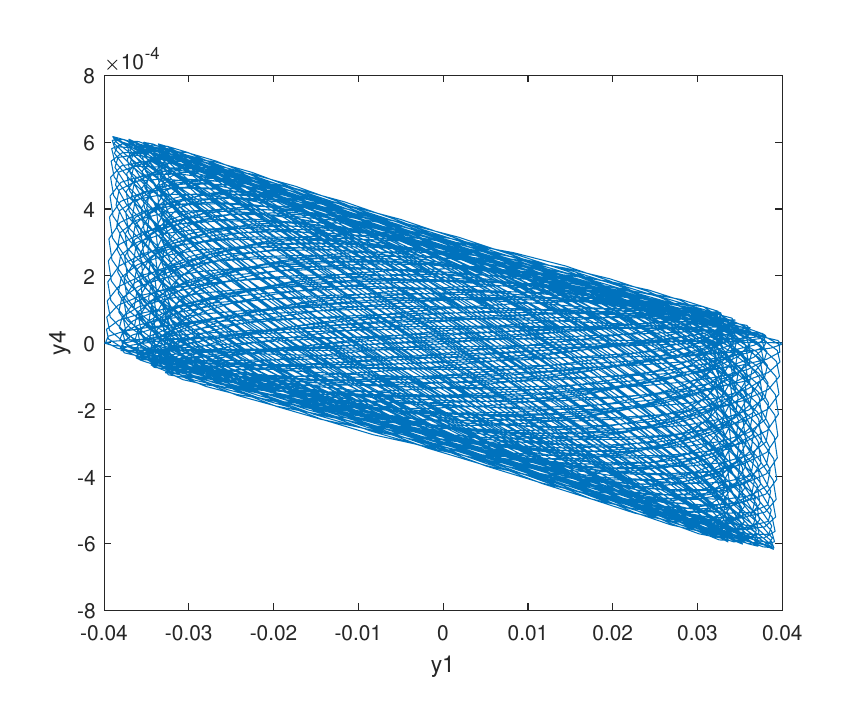}
	\end{minipage}
	
	\begin{minipage}{0.3\linewidth}
		\centering
		\includegraphics[width=55mm]{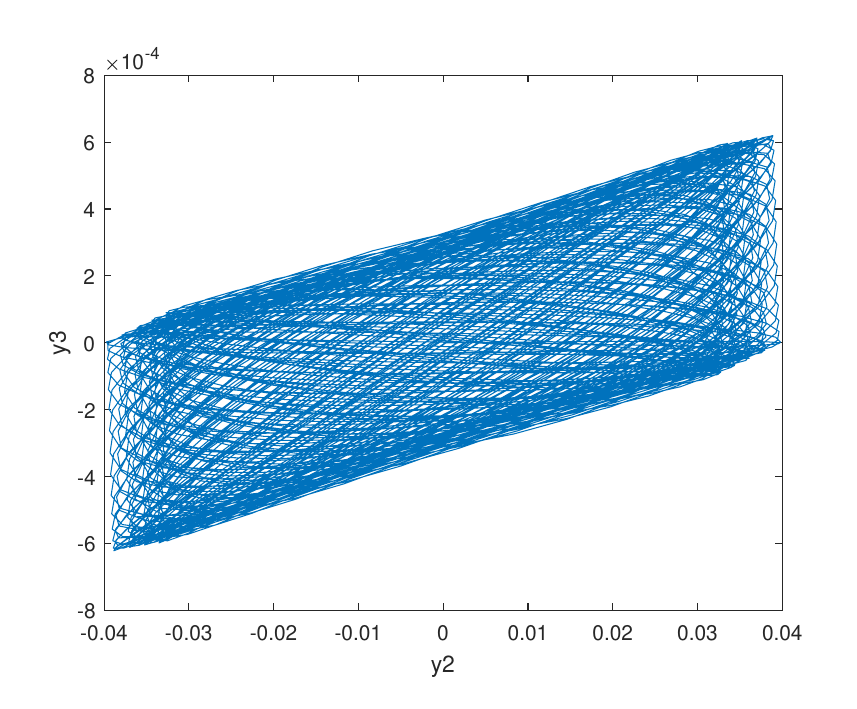}
	\end{minipage}
	\begin{minipage}{0.3\linewidth}
		\centering
		\includegraphics[width=55mm]{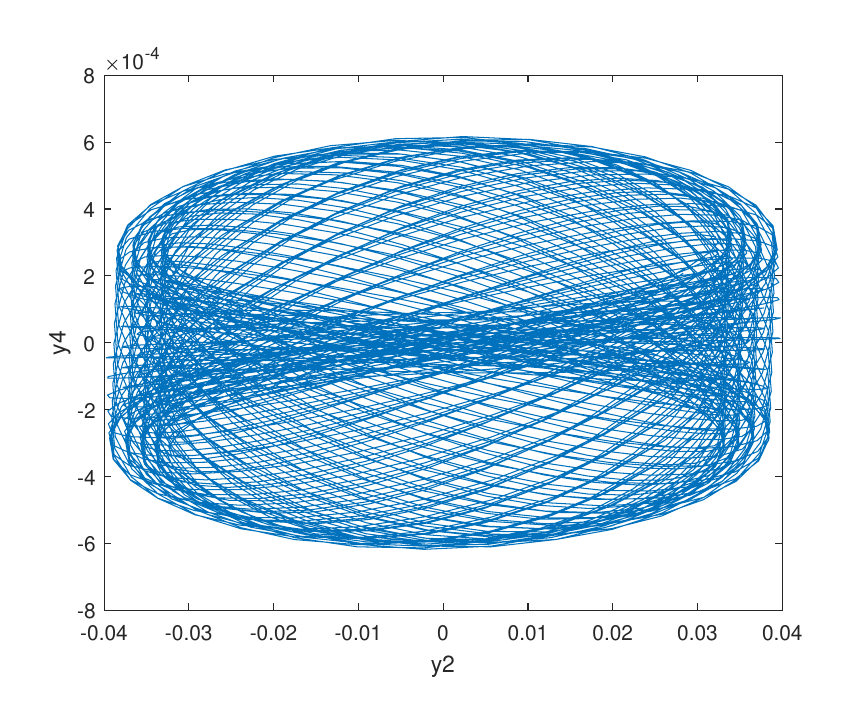}
	\end{minipage}
	\begin{minipage}{0.3\linewidth}
		\centering
		\includegraphics[width=55mm]{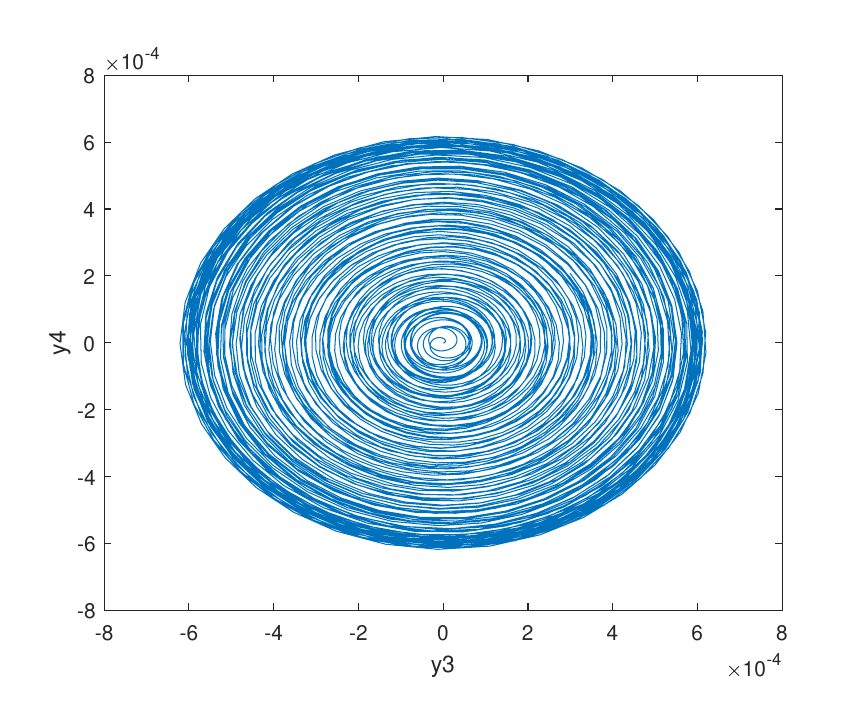}
	\end{minipage}
	\caption{When $(\alpha_{1},\alpha_{2})=(-0.00006,-0.00024)$ and $t\in \left[ 0,100 \right]$, the phase portraits of system \eqref{last}.}
	\label{Figure3-cm}
\end{figure}

\section*{Acknowledgments}
\par
The work was supported by the National Natural Science Foundation of China (No. 11971163).

\section*{Appendix}\label{secA}
\subsection*{A1. The proof of Lemma \ref{lemma1}}
\begin{proof}
	Since $\varepsilon^{2}+4a^{2}>4$, we have $\omega_{1} \neq \omega_{2}$. 
	Without loss of generality, we may consider that $\omega_{1} > \omega_{2}$. 
	we only need to prove that $\omega_{1} / \omega_{2}$ cannot be $2$ or $3$. 
	If $\omega_{1}=2\omega_{2}$, then it follows from equation \eqref{ceqs1} that
	\begin{align*}
		2\omega_{2}&=\omega_{1}=a \sin(\omega_{1} \tau)=a \sin(2\omega_{2} \tau)=2a \sin(\omega_{2} \tau) \cos(\omega_{2} \tau)=2\omega_{2} \cos(\omega_{2} \tau).
	\end{align*}
	It implies $\cos(\omega_{2} \tau)=1$, i.e., $\omega_{2}=a\sin(\omega_{2} \tau)=0$, which leads to a contradiction. 
	If $\omega_{1}=3\omega_{2}$, then, similarly,
	\begin{align*}
		3\omega_{2}&=\omega_{1}=a \sin(3\omega_{2} \tau)=a \sin(\omega_{2} \tau) \cdot (3-4\sin^{2}(\omega_{2} \tau))=3\omega_{2}-4\omega_{2}\sin^{2}(\omega_{2} \tau),
	\end{align*}
	i.e., $\omega_{2}\sin^{2}(\omega_{2} \tau)=0$. This contradicts the fact that $\omega_{2} \neq 0$.
\end{proof}

\subsection*{A2. Some exact expressions of $W(z,\theta)$ and $g(z,\alpha)$}
Let
\begin{align*}
	W(z,\theta)=\sum_{j+k+l+m \geq 2} w_{jklm}(\theta,\alpha) z_{1}^{j} z_{2}^{k} z_{3}^{l} z_{4}^{m}.
\end{align*}
Since $f^{\prime\prime}(0)=0$ in \eqref{fulleq}, it is apparent from the equation \eqref{cen-mani-ode} that $w_{jklm}(\theta,\alpha)=0$ if $j+k+l+m=2$. 
When $j+k+l+m=3$, solving the equation \eqref{cen-mani-ode}, we obtain 
\begin{align*}
	w_{jklm}(\theta,0)=e^{((j-k)\omega_{1}+(l-m)\omega_{2}){\rm i}\tau_{0}\theta}\left( \int_{0}^{\theta}e^{((j-k)\omega_{1}+(l-m)\omega_{2}){\rm i}\tau_{0}t}\Phi(t)\Psi(0)G_{jklm}\mathrm{d}t+c_{jklm}\right),
\end{align*}
with
\begin{small}
\begin{align*}
	c_{jklm}=&\left( ((j-k)\omega_{1}+(l-m)\omega_{2}){\rm i}\tau_{0}E-\tau_{0}M(0)-\tau_{0} e^{-((j-k)\omega_{1}+(l-m)\omega_{2}){\rm i}\tau_{0}} N(0) \right)^{-1}\\
	&\left( (E-\Phi(0)\Psi(0))G_{jklm}-\tau_{0} e^{-((j-k)\omega_{1}+(l-m)\omega_{2}){\rm i}\tau_{0}} N(0) \int_{-1}^{0}e^{((j-k)\omega_{1}+(l-m)\omega_{2}){\rm i}\tau_{0}t}\Phi(t)\Psi(0)G_{jklm}\mathrm{d}t \right),
\end{align*}
\end{small}
where $E$ is the $2 \times 2$ unit matrix and $G_{jklm}$ is the coefficient matrix of the term $z_{1}^{j} z_{2}^{k} z_{3}^{l} z_{4}^{m}$ in $G(\Phi z,0)$.
\par
Set
\begin{align*}
	g_{i}(z,\alpha)=\sum_{j+k+l+m \geq 1} g_{i,jklm}(\alpha) z_{1}^{j} z_{2}^{k} z_{3}^{l} z_{4}^{m}, \qquad i=1,2,3,4,
\end{align*}
where $g_{i}(z,\alpha)$ is the $i$th component of $g(z,\alpha)$.
Some tedious manipulation yields
\begin{align*}
	g_{2,jklm}(\alpha)=&g_{1,kjml}(\alpha), \qquad g_{4,jklm}(\alpha)=g_{3,kjml}(\alpha),\\
	g_{1,1000}(\alpha)=&D_{1} \tau_{0} ( a e^{-{\rm i} \tau_{0} \omega_{1}} + {\rm i} \omega_{1} ) \alpha_{1}
	+D_{1} ( a \varepsilon_{0} e^{-{\rm i} \tau_{0} \omega_{1}} - \omega_{1}^{2}-1 ) \alpha_{2}
	+O(\left\| \alpha \right\| ^{2}),\\
	g_{1,0100}(\alpha)=&D_{1} \tau_{0} ( a e^{{\rm i} \tau_{0} \omega_{1}} - {\rm i} \omega_{1} ) \alpha_{1}
	+D_{1} ( a \varepsilon_{0} e^{{\rm i} \tau_{0} \omega_{1}} + \omega_{1}^{2}-1 ) \alpha_{2}
	+O(\left\| \alpha \right\| ^{2}),\\
	g_{1,0010}(\alpha)=&D_{1} \tau_{0} ( a e^{-{\rm i} \tau_{0} \omega_{2}} + {\rm i} \omega_{2} ) \alpha_{1}
	+D_{1} ( a \varepsilon_{0} e^{-{\rm i} \tau_{0} \omega_{2}} - \omega_{1}\omega_{2}-1 ) \alpha_{2}
	+O(\left\| \alpha \right\| ^{2}),\\
	g_{1,0001}(\alpha)=&D_{1} \tau_{0} ( a e^{{\rm i} \tau_{0} \omega_{2}} - {\rm i} \omega_{2} ) \alpha_{1}
	+D_{1} ( a \varepsilon_{0} e^{{\rm i} \tau_{0} \omega_{2}} + \omega_{1}\omega_{2}-1 ) \alpha_{2}
	+O(\left\| \alpha \right\| ^{2}),\\
	g_{3,1000}(\alpha)=&D_{2} \tau_{0} ( a e^{-{\rm i} \tau_{0} \omega_{1}} + {\rm i} \omega_{1} ) \alpha_{1}
	+D_{2} ( a \varepsilon_{0} e^{-{\rm i} \tau_{0} \omega_{1}} - \omega_{1}\omega_{2}-1 ) \alpha_{2}
	+O(\left\| \alpha \right\| ^{2}),\\
	g_{3,0100}(\alpha)=&D_{2} \tau_{0} ( a e^{{\rm i} \tau_{0} \omega_{1}} - {\rm i} \omega_{1} ) \alpha_{1}
	+D_{2} ( a \varepsilon_{0} e^{{\rm i} \tau_{0} \omega_{1}} + \omega_{1}\omega_{2}-1 ) \alpha_{2}
	+O(\left\| \alpha \right\| ^{2}),\\
	g_{3,0010}(\alpha)=&D_{2} \tau_{0} ( a e^{-{\rm i} \tau_{0} \omega_{2}} + {\rm i} \omega_{2} ) \alpha_{1}
	+D_{2} ( a \varepsilon_{0} e^{-{\rm i} \tau_{0} \omega_{2}} - \omega_{2}^{2}-1 ) \alpha_{2}
	+O(\left\| \alpha \right\| ^{2}),\\
	g_{3,0001}(\alpha)=&D_{2} \tau_{0} ( a e^{{\rm i} \tau_{0} \omega_{2}} - {\rm i} \omega_{2} ) \alpha_{1}
	+D_{2} ( a \varepsilon_{0} e^{{\rm i} \tau_{0} \omega_{2}} + \omega_{2}^{2}-1 ) \alpha_{2}
	+O(\left\| \alpha \right\| ^{2}),
\end{align*}
and $g_{i,jklm}(0)=0$, for $i=1,3$, $j+k+l+m=2$,
\begin{align*}
	g_{1,3000}(0)=&D_{1}\tau_{0}\varepsilon_{0} (b e^{-3{\rm i}\tau_{0}\omega_{1}} -{\rm i}\omega_{1}),&
	g_{3,3000}(0)=&D_{2}\tau_{0}\varepsilon_{0} (b e^{-3{\rm i}\tau_{0}\omega_{1}} -{\rm i}\omega_{1}),\\
	g_{1,0300}(0)=&D_{1}\tau_{0}\varepsilon_{0} (b e^{3{\rm i}\tau_{0}\omega_{1}} +{\rm i}\omega_{1}),&
	g_{3,0300}(0)=&D_{2}\tau_{0}\varepsilon_{0} (b e^{3{\rm i}\tau_{0}\omega_{1}} +{\rm i}\omega_{1}),\\
	g_{1,0030}(0)=&D_{1}\tau_{0}\varepsilon_{0} (b e^{-3{\rm i}\tau_{0}\omega_{2}} -{\rm i}\omega_{2}),&
	g_{3,0030}(0)=&D_{2}\tau_{0}\varepsilon_{0} (b e^{-3{\rm i}\tau_{0}\omega_{2}} -{\rm i}\omega_{2}),\\
	g_{1,0003}(0)=&D_{1}\tau_{0}\varepsilon_{0} (b e^{3{\rm i}\tau_{0}\omega_{2}} +{\rm i}\omega_{2}),&
	g_{3,0003}(0)=&D_{2}\tau_{0}\varepsilon_{0} (b e^{3{\rm i}\tau_{0}\omega_{2}} +{\rm i}\omega_{2}),\\
	g_{1,2100}(0)=&D_{1}\tau_{0}\varepsilon_{0} (3b e^{-{\rm i}\tau_{0}\omega_{1}} -{\rm i}\omega_{1}),&
	g_{3,2100}(0)=&D_{2}\tau_{0}\varepsilon_{0} (3b e^{-{\rm i}\tau_{0}\omega_{1}} -{\rm i}\omega_{1}),\\
	g_{1,1200}(0)=&D_{1}\tau_{0}\varepsilon_{0} (3b e^{{\rm i}\tau_{0}\omega_{1}} +{\rm i}\omega_{1}),&
	g_{3,1200}(0)=&D_{2}\tau_{0}\varepsilon_{0} (3b e^{{\rm i}\tau_{0}\omega_{1}} +{\rm i}\omega_{1}),\\
	g_{1,2010}(0)=&D_{1}\tau_{0}\varepsilon_{0} (3b e^{-{\rm i}\tau_{0}(2\omega_{1}+\omega_{2})} -{\rm i}(2\omega_{1}+\omega_{2})),&
	g_{3,2010}(0)=&D_{2}\tau_{0}\varepsilon_{0} (3b e^{-{\rm i}\tau_{0}(2\omega_{1}+\omega_{2})} -{\rm i}(2\omega_{1}+\omega_{2})),\\
	g_{1,1020}(0)=&D_{1}\tau_{0}\varepsilon_{0} (3b e^{-{\rm i}\tau_{0}(\omega_{1}+2\omega_{2})} -{\rm i}(\omega_{1}+2\omega_{2})),&
	g_{3,1020}(0)=&D_{2}\tau_{0}\varepsilon_{0} (3b e^{-{\rm i}\tau_{0}(\omega_{1}+2\omega_{2})} -{\rm i}(\omega_{1}+2\omega_{2})),\\
	g_{1,2001}(0)=&D_{1}\tau_{0}\varepsilon_{0} (3b e^{-{\rm i}\tau_{0}(2\omega_{1}-\omega_{2})} -{\rm i}(2\omega_{1}-\omega_{2})),&
	g_{3,2001}(0)=&D_{2}\tau_{0}\varepsilon_{0} (3b e^{-{\rm i}\tau_{0}(2\omega_{1}-\omega_{2})} -{\rm i}(2\omega_{1}-\omega_{2})),\\
	g_{1,1002}(0)=&D_{1}\tau_{0}\varepsilon_{0} (3b e^{-{\rm i}\tau_{0}(\omega_{1}-2\omega_{2})} -{\rm i}(\omega_{1}-2\omega_{2})),&
	g_{3,1002}(0)=&D_{2}\tau_{0}\varepsilon_{0} (3b e^{-{\rm i}\tau_{0}(\omega_{1}-2\omega_{2})} -{\rm i}(\omega_{1}-2\omega_{2})),\\
	g_{1,0210}(0)=&D_{1}\tau_{0}\varepsilon_{0} (3b e^{{\rm i}\tau_{0}(2\omega_{1}-\omega_{2})} +{\rm i}(2\omega_{1}-\omega_{2})),&
	g_{3,0210}(0)=&D_{2}\tau_{0}\varepsilon_{0} (3b e^{{\rm i}\tau_{0}(2\omega_{1}-\omega_{2})} +{\rm i}(2\omega_{1}-\omega_{2})),\\
	g_{1,0120}(0)=&D_{1}\tau_{0}\varepsilon_{0} (3b e^{{\rm i}\tau_{0}(\omega_{1}-2\omega_{2})} +{\rm i}(\omega_{1}-2\omega_{2})),&
	g_{3,0120}(0)=&D_{2}\tau_{0}\varepsilon_{0} (3b e^{{\rm i}\tau_{0}(\omega_{1}-2\omega_{2})} +{\rm i}(\omega_{1}-2\omega_{2})),\\
	g_{1,0201}(0)=&D_{1}\tau_{0}\varepsilon_{0} (3b e^{{\rm i}\tau_{0}(2\omega_{1}+\omega_{2})} +{\rm i}(2\omega_{1}+\omega_{2})),&
	g_{3,0201}(0)=&D_{2}\tau_{0}\varepsilon_{0} (3b e^{{\rm i}\tau_{0}(2\omega_{1}+\omega_{2})} +{\rm i}(2\omega_{1}+\omega_{2})),\\
	g_{1,0102}(0)=&D_{1}\tau_{0}\varepsilon_{0} (3b e^{{\rm i}\tau_{0}(\omega_{1}+2\omega_{2})} +{\rm i}(\omega_{1}+2\omega_{2})),&
	g_{3,0102}(0)=&D_{2}\tau_{0}\varepsilon_{0} (3b e^{{\rm i}\tau_{0}(\omega_{1}+2\omega_{2})} +{\rm i}(\omega_{1}+2\omega_{2})),\\
	g_{1,0021}(0)=&D_{1}\tau_{0}\varepsilon_{0} (3b e^{-{\rm i}\tau_{0}\omega_{2}} -{\rm i}\omega_{2}),&
	g_{3,0021}(0)=&D_{2}\tau_{0}\varepsilon_{0} (3b e^{-{\rm i}\tau_{0}\omega_{2}} -{\rm i}\omega_{2}),\\
	g_{1,0012}(0)=&D_{1}\tau_{0}\varepsilon_{0} (3b e^{{\rm i}\tau_{0}\omega_{2}} +{\rm i}\omega_{2}),&
	g_{3,0012}(0)=&D_{2}\tau_{0}\varepsilon_{0} (3b e^{{\rm i}\tau_{0}\omega_{2}} +{\rm i}\omega_{2}),\\
	g_{1,1110}(0)=&2g_{1,0021}(0),&
	g_{3,1110}(0)=&2g_{3,0021}(0),\\
	g_{1,1101}(0)=&2g_{1,0012}(0),&
	g_{3,1101}(0)=&2g_{3,0012}(0),\\
	g_{1,1011}(0)=&2g_{1,2100}(0),&
	g_{3,1011}(0)=&2g_{3,2100}(0),\\
	g_{1,0111}(0)=&2g_{1,1200}(0),&
	g_{3,0111}(0)=&2g_{3,1200}(0),
\end{align*}
\begin{align*}	
	g_{1,3200}(0)=&D_{1} \tau_{0} \varepsilon_{0} (10 d e^{-{\rm i} \tau_{0} \omega_{1}} 
	+3 b w_{1,1200}(-1) e^{-2 {\rm i} \tau_{0} \omega_{1}}
	+3 b w_{1,3000}(-1) e^{ 2 {\rm i} \tau_{0} \omega_{1}}
	+6 b w_{1,2100}(-1)\\
	&-2 {\rm i} w_{1,1200}(0) \omega_{1}
	+2 {\rm i} w_{1,3000}(0) \omega_{1}
	-w_{2,1200}(0)-2 w_{2,2100}(0)-w_{2,3000}(0)), \\
	g_{1,2111}(0)=&D_{1} \tau_{0} \varepsilon_{0}(60 d e^{-{\rm i} \tau_{0} \omega_{1}} 
	+3 b w_{1,0111}(-1) e^{-2 {\rm i} \tau_{0} \omega_{1}}
	+6 b w_{1,1101}(-1) e^{-{\rm i} \tau_{0} (\omega_{1}+\omega_{2})}\\
	&+6 b w_{1,1110}(-1) e^{-{\rm i} \tau_{0} (\omega_{1}-\omega_{2})}
	+6 b w_{1,2001}(-1) e^{{\rm i} \tau_{0} (\omega_{1}-\omega_{2})}
	+6 b w_{1,2010}(-1) e^{{\rm i} \tau_{0} (\omega_{1}+\omega_{2})}\\
	&+6 b w_{1,1011}(-1)
	+6 b w_{1,2100}(-1)
	+2 {\rm i} w_{1,2010}(0) \omega_{1}
	+2 {\rm i} w_{1,2001}(0) \omega_{1}
	-2 {\rm i} w_{1,0111}(0) \omega_{1}\\
	&-2 {\rm i} w_{1,1101}(0) \omega_{1}
	-2 {\rm i} w_{1,1110}(0) \omega_{1}
	+2 {\rm i} w_{1,1110}(0) \omega_{2}
	+2 {\rm i} w_{1,2010}(0) \omega_{2}
	-2 {\rm i} w_{1,2001}(0) \omega_{2}\\
	&-2 {\rm i} w_{1,1101}(0) \omega_{2}
	-w_{2,0111}(0)-2 w_{2,1011}(0)-2 w_{2,1101}(0)-2 w_{2,1110}(0)\\
	&-2 w_{2,2001}(0)-2 w_{2,2010}(0)-2 w_{2,2100}(0) ), \\
	g_{1,1022}(0)=&D_{1} \tau_{0} \varepsilon_{0}(30 d e^{-{\rm i} \tau_{0} \omega_{1}}
	+6 b w_{1,0012}(-1) e^{-{\rm i} \tau_{0} (\omega_{1}+\omega_{2})}
	+6 b w_{1,0021}(-1) e^{-{\rm i} \tau_{0} (\omega_{1}-\omega_{2})}\\
	&+3 b w_{1,1002}(-1) e^{-2 {\rm i} \tau_{0} \omega_{2}}
	+3 b w_{1,1020}(-1) e^{2 {\rm i} \tau_{0} \omega_{2}}
	+6 b w_{1,1011}(-1)
	-2 {\rm i} w_{1,0012}(0) \omega_{1}\\
	&-2 {\rm i} w_{1,0021}(0) \omega_{1}
	+2 {\rm i} w_{1,0021}(0) \omega_{2}
	+2 {\rm i} w_{1,1020}(0) \omega_{2}
	-2 {\rm i} w_{1,0012}(0) \omega_{2}
	-2 {\rm i} w_{1,1002}(0) \omega_{2}\\
	&-w_{2,1002}(0)-w_{2,1020}(0)-2 w_{2,0012}(0)-2 w_{2,0021}(0)-2 w_{2,1011}(0)), \\
\end{align*}
\begin{align*}
	g_{3,2210}(0)=&D_{2} \tau_{0} \varepsilon_{0}(30 d e^{-{\rm i} \tau_{0} \omega_{2}}
	+6 b w_{1,2100}(-1) e^{{\rm i} \tau_{0} (\omega_{1}-\omega_{2})}
	+6 b w_{1,1200}(-1) e^{-{\rm i} \tau_{0} (\omega_{1}+\omega_{2})}\\
	&+3 b w_{1,2010}(-1) e^{2 {\rm i} \tau_{0} \omega_{1}}
	+3 b w_{1,0210}(-1) e^{-2 {\rm i} \tau_{0} \omega_{1}}
	+6 b w_{1,1110}(-1)
	+2 {\rm i} w_{1,2010}(0) \omega_{1}\\
	&+2 {\rm i} w_{1,2100}(0) \omega_{1}
	-2 {\rm i} w_{1,0210}(0) \omega_{1}
	-2 {\rm i} w_{1,1200}(0) \omega_{1}
	-2 {\rm i} w_{1,1200}(0) \omega_{2}
	-2 {\rm i} w_{1,2100}(0) \omega_{2}\\
	&-w_{2,0210}(0)-w_{2,2010}(0)-2 w_{2,1110}(0)-2 w_{2,1200}(0)-2 w_{2,2100}(0)), \\
	g_{3,1121}(0)=&D_{2} \tau_{0} \varepsilon_{0}(60 d e^{-{\rm i} \tau_{0} \omega_{2}} 
	+3 b w_{1,1101}(-1) e^{-2 {\rm i} \tau_{0} \omega_{2}}
	+6 b w_{1,0111}(-1) e^{-{\rm i} \tau_{0} (\omega_{1}+\omega_{2})}\\
	&+6 b w_{1,0120}(-1) e^{-{\rm i} \tau_{0} (\omega_{1}-\omega_{2})}
	+6 b w_{1,1011}(-1) e^{{\rm i} \tau_{0} (\omega_{1}-\omega_{2})}
	+6 b w_{1,1020}(-1) e^{{\rm i} \tau_{0} (\omega_{1}+\omega_{2})}\\
	&+6 b w_{1,0021}(-1)
	+6 b w_{1,1110}(-1)
	+2 {\rm i} w_{1,1020}(0) \omega_{1}
	+2 {\rm i} w_{1,1011}(0) \omega_{1}
	-2 {\rm i} w_{1,0111}(0) \omega_{1}\\
	&-2 {\rm i} w_{1,0120}(0) \omega_{1}
	+2 {\rm i} w_{1,0120}(0) \omega_{2}
	+2 {\rm i} w_{1,1020}(0) \omega_{2}
	-2 {\rm i} w_{1,1011}(0) \omega_{2}
	-2 {\rm i} w_{1,1101}(0) \omega_{2}\\
	&-2 {\rm i} w_{1,0111}(0) \omega_{2}
	-2 w_{2,0021}(0)-2 w_{2,0111}(0)-2 w_{2,0120}(0)-2 w_{2,1011}(0)\\
	&-2 w_{2,1020}(0)-2 w_{2,1110}(0)-w_{2,1101}(0) ), \\
	g_{3,0032}(0)=&D_{2} \tau_{0} \varepsilon_{0}(10 d e^{-{\rm i} \tau_{0} \omega_{2}}
	+3 b w_{1,0012}(-1) e^{-2 {\rm i} \tau_{0} \omega_{2}}
	+3 b w_{1,0030}(-1) e^{2 {\rm i} \tau_{0} \omega_{2}}
	+6 b w_{1,0021}(-1)\\
	&+2 {\rm i} w_{1,0030}(0) \omega_{2}
	-2 {\rm i} w_{1,0012}(0) \omega_{2}
	-w_{2,0030}(0)-w_{2,0012}(0)-2 w_{2,0021}(0)).
\end{align*}

\subsection*{A3. The value of $A_{ij}$}
For fixed $a=1$, $b=\frac{1}{6}$, $c=d=0$, when $\varepsilon_{0} \approx 0.2533$ and $\tau_{0}=2.5\pi$, we obtain
\begin{align*}
	A_{11}&=722.299-{\rm i}12679.4,
	&A_{12}&=3233.98-{\rm i}62094.9,
	&A_{13}&=2179.59-{\rm i}30955.1,\\
	A_{21}&=-2160.96-{\rm i}31242.7,
	&A_{22}&=-3244.20+{\rm i}62156.1,
	&A_{23}&=-731.82+{\rm i}12190.3.
\end{align*}

\subsection*{A4. Expressions in \eqref{4}}
\begin{align*}
	d_{11}=(2-\gamma-\gamma\sigma)\zeta_{2}
	-\mathrm{i}(\gamma+2) \sqrt{\gamma  \sigma -1}\zeta_{2}
	+(-2 (\gamma  \sigma +\gamma -1)+\mathrm{i}\frac{-\gamma ^2 \sigma -3 \gamma  \sigma +2 \gamma +2}{\sqrt{\gamma  \sigma -1}})\alpha_{3} +O(\alpha_{3}^2)+O(\epsilon),
\end{align*}
\begin{align*}
	d_{12}=&(-2 \gamma
	-\mathrm{i}\frac{2 \gamma  (1-\sigma ) \sqrt{\gamma  \sigma -1}}{1-\gamma  \sigma })\alpha_{3}
	+\zeta_{2}^{2}(2 \left(-\gamma  \Gamma _1-\Gamma _3 \sigma +\Gamma _2\right)\\
	&+\mathrm{i}\frac{2 \left(-\gamma ^2 \Sigma _1+\gamma  \Gamma _1-\gamma  \sigma  \Sigma _3+\gamma  \Sigma _2+\Gamma _3 \sigma -\Gamma _2\right)}{\sqrt{\gamma  \sigma -1}})\epsilon+O(\alpha_{3}^2)+O(\epsilon^2)+O(\epsilon \alpha_{3}),
\end{align*}
\begin{align*}
	d_{13}=-(2-\gamma-\gamma\sigma)\zeta_{2}
	+\mathrm{i}(\gamma -2) \sqrt{\gamma  \sigma -1}\zeta_{2}
	+(2 (\gamma  \sigma -1)+\mathrm{i}\frac{\gamma ^2 \sigma -3 \gamma  \sigma +2}{\sqrt{\gamma  \sigma -1}})\alpha_{3}
	+O(\alpha_{3}^2)+O(\epsilon),
\end{align*}
\begin{align*}
	D_{12}=&(-2 \gamma  \left(-3 \gamma  \Gamma _1-\gamma  \Sigma _1-\Gamma _3 \sigma +2 \Gamma _2-3 \sigma  \Sigma _3+2 \Sigma _2\right)\\
	&+\mathrm{i}(\frac{2 \gamma  \left(\gamma  \Gamma _2 \sigma -3 \gamma  \Gamma _1-\gamma  \sigma  \Sigma _2+3 \gamma  \Sigma _1-3 \Gamma _3 \sigma +2 \Gamma _2+3 \sigma  \Sigma _3-2 \Sigma _2\right)}{\sqrt{\gamma  \sigma -1}}))\zeta_{2}^{2}+O(\alpha_{3})+O(\epsilon).
\end{align*}

\subsection*{A5. Expressions in \eqref{6}}
\begin{align*}
	l_{21}(\alpha_{3})=\frac{3\alpha_{3}+\beta_{3}{\rm i}}{\alpha_{3}^{2}+\beta_{3}^{2}} d_{11} d_{12} + \frac{\alpha_{3}-\beta_{3}{\rm i}}{\alpha_{3}^{2}+\beta_{3}^{2}} \left| d_{12} \right|^{2} + \frac{2\alpha_{3}-6\beta_{3}{\rm i}}{\alpha_{3}^{2}+9\beta_{3}^{2}} \left| d_{13} \right|^{2} + \epsilon D_{12},
\end{align*}
\begin{align*}
	l_{32}(0)=-\frac{{\rm i}}{3 \beta_{3}^{3}}\left| d_{11} \right|^{2}\left| d_{13} \right|^{2}-\frac{8{\rm i}}{27 \beta_{3}^{3}}\left| d_{13} \right|^{4}+O(\epsilon),
\end{align*}
\begin{align*}
	l_{43}(0)=-\frac{4{\rm i}}{45 \beta_{3}^{5}}\left| d_{13} \right|^{2}{\rm Re}\left\lbrace d_{11}^{3}d_{13}\right\rbrace -\frac{{\rm i}}{5 \beta_{3}^{5}}\left| d_{11} \right|^{4}\left| d_{13} \right|^{2} -\frac{37{\rm i}}{90 \beta_{3}^{5}}\left| d_{11} \right|^{2}\left| d_{13} \right|^{4} -\frac{20{\rm i}}{81 \beta_{3}^{5}}\left| d_{13} \right|^{6}+O(\epsilon),
\end{align*}

\end{document}